\documentclass{amsart}
\usepackage[utf8]{inputenc}
\usepackage[english]{babel}
\usepackage[all]{xy}
\usepackage{graphicx}
\usepackage{amsmath,amsthm,bm,mathrsfs,amscd,tikz}
\usepackage{amsfonts,amssymb,mathrsfs}
\usepackage{verbatim}
\usepackage{float}
\usepackage{enumerate}
\RequirePackage{color}
\usepackage{xcolor}
\usepackage{soul}

\theoremstyle{plain}
\newtheorem{theorem}{Theorem}[section]
\newtheorem{proposition}[theorem]{Proposition}
\newtheorem{corollary}[theorem]{Corollary}
\newtheorem{lemma}[theorem]{Lemma}

\newtheorem{remark}[theorem]{Remark}

\theoremstyle{definition}
\newtheorem{definition}[theorem]{Definition}
\newtheorem{example}[theorem]{Example}

\newcommand{\Ext}{\operatorname{Ext}}
\newcommand{\Hom}{\operatorname{Hom}}

\newcommand{\mmod}{{\rm mod}\,}

\newcommand{\Z}{\mathbb{Z}}
\newcommand{\C}{\mathcal{C}}

\newcommand{\T}{\mathcal{T}}

\def\id{\hbox{1\hskip -3pt {\sf I}}}

\author{Charles Paquette}\address{Department of Mathematics and Computer Science, Royal Military College of Canada,
Kingston, ON K7K 7B4, Canada}
\email{charles.paquette.math@gmail.com}
\author{Emine Yildirim}
\address{Department of Mathematics and Statistics, Queen's University, Kingston, ON K7L 3N6, Canada}
\email{emineyyildirim@gmail.com}

\subjclass[2010]{18E30, 16G20}
\keywords{cluster category, cluster character, cluster-tilting, cluster algebra}

\thanks{ The first author was supported by the Natural Sciences and Engineering Research Council of Canada and by Canadian Defence Academy Research Programme. We would like to thank the anonymous referee for carefully reading the paper and the suggestions made which led to an improved version of the paper.}

\begin{document}
\title{Completions of discrete cluster categories of type $\mathbb{A}$}

\maketitle 

\begin{abstract}
    We complete the discrete cluster categories of type $\mathbb{A}$ as defined by Igusa and Todorov, by embedding such a discrete cluster category inside a larger one, and then taking a certain Verdier quotient. The resulting category is a Hom-finite Krull-Schmidt triangulated category containing the discrete cluster category as a full subcategory. The objects and Hom-spaces in this new category can be described geometrically, even though the category is not $2$-Calabi-Yau and Ext-spaces are not always symmetric. We describe all cluster-tilting subcategories. Given such a subcategory, we define a cluster character that takes values in a ring with infinitely many indeterminates. Our cluster character is new in that it takes into account infinite dimensional sub-representations of infinite dimensional ones. We show that it satisfies the multiplication formula and also the exchange formula, provided that the objects being exchanged satisfy some local Calabi-Yau conditions.
\end{abstract}

\section*{Introduction}

Discrete cluster categories of type $\mathbb{A}$ have been introduced by Igusa and Todorov in \cite{IT} as a nice class of $2$-Calabi-Yau Hom-finite triangulated categories generalizing the classical cluster categories of type $\mathbb{A}_n$ introduced by Caldero-Chapoton and Schiffler in \cite{CCS} and also by Buan, Marsh, Reineke, Reiten and Todorov in \cite{BMRRT} for general acyclic and finite quivers. Discrete cluster categories of type $\mathbb{A}$ are discrete analogues of the continuous cluster categories also introduced by Igusa and Todorov in \cite{IT,IT2}. Let $S$ be the disk and consider $M$ a discrete set of marked points on the boundary having finitely many two-sided accumulation points. To the pair $(S,M)$, Igusa and Todorov define the corresponding cluster category $\mathcal{C}_{(S,M)}$ of type $\mathbb{A}$ where indecomposable objects, up to isomorphisms, are in bijection with arcs of $(S,M)$ and the Ext-space between indecomposable objects $X,Y$ is non-zero (and one dimensional) if and only if the arcs corresponding to $X,Y$ cross. 

\medskip

In this paper, we extend any discrete cluster category $\mathcal{C}_{(S,M)}$ of type $\mathbb{A}$ to a larger Hom-finite triangulated category $\overline{\mathcal{C}}_{(S,M)}$. Although not $2$-Calabi-Yau, the category $\overline{\mathcal{C}}_{(S,M)}$ has the following combinatorial interpretation. Indecomposable objects correspond to arcs of $(S,\overline{M})$ where $\overline{M}$ is the closure of $M$. When two arcs cross, the corresponding Ext-space is one-dimensional. The converse, however, is not true in general if one of $X,Y$ corresponds to a limit arc. Our construction uses Verdier quotients and calculus of fractions. Once the categories have been introduced and studied, we give a complete description of the cluster-tilting subcategories of $\overline{\mathcal{C}}_{(S,M)}$. In particular, they correspond to the completion of the cluster-tilting subcategories as described in \cite{GHJ}, but without leapfrog configurations. Given such a cluster-tilting subcategory $\mathcal{T}$ of $\overline{\mathcal{C}}_{(S,M)}$, we can define a cluster character $X^{\mathcal{T}}$ on the objects of $\overline{\mathcal{C}}_{(S,M)}$ that takes value in a "ring" with infinitely many indeterminates. In order to define this, we need to define Grassmannian quiver varieties of infinite dimensional representations with possibly infinite dimensional subrepresentations. We show that our cluster character satisfies the multiplication formula and, for objects corresponding to ordinary (not limit) arcs, the exchange formula also holds. We give many examples and illustrations all over the sections. 

\medskip

Cluster combinatorics in the setting of Riemann surfaces with infinitely many marked points have been studied extensively in the recent years; see for instance \cite{BG, CF, HJ, LP}. In \cite{BG, CF}, accumulation points were considered as part of the set of marked points, and some infinite sequences of mutations were allowed in the cluster combinatorics. In \cite{CF}, infinite rank cluster algebras were introduced for $(S,\overline{M}, P)$ where $S$ is an oriented Riemann surface, $P$ a finite set of punctures in the interior of $S$, and $M$ is a discrete set of marked points with finitely many two-sided accumulations points, and where each boundary component has at least one point in $M$. The values of our cluster character on indecomposable objects in $\overline{\mathcal{C}}_{(S,M)}$ differ slightly from the cluster variables as defined by \c{C}anak\c{c}\i\ and Felikson for the surface $(S,\overline{M}).$  
In the paper~\cite{BG}, Baur-Gratz study in more details the cluster combinatorics of the completed infinity-gon, but where they consider two distinct one-sided accumulation points. Fisher~\cite{F} has also studied the completed infinity-gon but with a single two-sided accumulation point, where he has obtained a completion of the Igusa-Todorov discrete cluster category in that case. He also finds that Ext-spaces are not symmetric. His construction is done at the level of the discrete cluster category by adding new objects which correspond to homotopy colimits.  Finally, we would like to mention that the cluster category of this one-accumulation point case has also been recovered recently by August-Cheung-Faber-Gratz-Schroll in \cite{ACFGS} in the study of infinite-rank Grassmannians, and where the non-symmetric behavior was also noticed.

\section{The disk with marked points}

Throughout the paper, we let $S$ be a disk $D^1$ bounded by the circle $\partial S=S^1$. We fix the positive orientation to be counter-clockwise. We let $M$ denote an infinite set of points on the boundary $\partial S$ of $S$. The points in $M$ are called \emph{marked points}.

We say that a point $z$ on $\partial S$ is an \emph{accumulation point} (from $M$) if there is a sequence of pairwise distinct marked points in $M$ that converges to $z$. Convergence is with respect to the usual metric topology. An interval $[x,y]$ on $\partial S$ is the set of points in $\partial S$ from $x$ to $y$ following the counter-clockwise orientation. We will call $x$ the \emph{left side}  of the interval and $y$ its \emph{right side}. Now, the notions of convergence on the left and convergence on the right are defined naturally.

An accumulation point $z$ is called \emph{two-sided} if there are two sequences $\{m_i\}_{i \ge 1}$ and $\{m_i'\}_{i \ge 1}$ of pairwise distinct marked points that both converge to $z$, on the left and on the right, respectively. We let ${\rm acc}(M)$ denote the set of all accumulation points. An \emph{arc} is a continuous curve without self-intersections, from a point of $M$ to another point of $M$. Two arcs having the same endpoints are identified.

\begin{definition}~\label{lim-arc}
If we have a continuous curve from $M\cup{\rm acc}(M)$ to $M\cup{\rm acc}(M)$ for which one or both of its endpoints is an accumulation point, then we call this curve a \emph{limit arc}.
\end{definition} 
Again, two limit arcs having the same endpoints are identified. It is convenient to identify an arc (or a limit arc) with a set $\{a,b\}$ where $a,b$ are distinct in $M$ (or in $M\cup{\rm acc}(M)$, respectively). We say that $M$ is \emph{discrete} if for all $m\in M$ there exists an open neighborhood $N_m$ of $m$ on $\partial S$ such that $N_m\cap M=\{m\}$. We say that $M$ is \emph{closed} if it contains all the accumulation points, that is, ${\rm acc}(M) \subseteq M$. By a \emph{regular} marked point, we mean a marked point which is not an accumulation point. 

\begin{remark}(1) Note that ${\rm acc}(M)$ is non-empty if and only if $M$ is infinite.
\begin{enumerate}
\item[$(2)$] Observe that $M$ is discrete if and only if ${\rm acc}(M) \cap M = \emptyset$.
\item[$(3)$] If $M$ is discrete, then $M$ is countable. \end{enumerate}
\end{remark}

Let $m$ be a marked point on $\partial S$.
If there is another marked point $m'$ on $\partial S$ which follows $m$ in the orientation of $\partial S$ such that the open interval $(m, m')$ on $\partial S$ contains no marked point, then we let $m' = m^+$. Similarly, if there is another marked point $m''$ on $\partial S$ which follows $m$ in the opposite orientation such that the interval $(m'', m)$ on $\partial S$ contains no marked point, then we let $m'' = m^-$. Hence, we get two functions $\sigma$ and $\sigma^-$ that are partially defined on $M$ such that $\sigma(m) = m^-$ whenever $m^-$ is defined and $\sigma^-(m)=m^+$ whenever $m^+$ is defined. Note that these functions are inverses of each other on their respective domains of definition.

We will see later that the functions $\sigma$ and $\sigma^-$ will play a crucial role in the construction of our categories. Therefore, the lemma below explains why having accumulation points that are two-sided is an important property.

\begin{lemma}
The functions $\sigma, \sigma^-$ are defined on all of $M$ if and only if $M$ is discrete and all accumulation points are two-sided.
\end{lemma}

We also remark that $\sigma$ has no fixed points.

\section{Discrete cluster categories of type $\mathbb{A}$}~\label{secondsection} In this section, we recollect the results we need about the so-called discrete cluster categories of type $\mathbb{A}$, as defined by Igusa and Todorov in \cite{IT} (see also \cite{GHJ}). We use the notations and setting of the previous section.
Assume that $\sigma, \sigma^-$ are defined everywhere and are bijective. Therefore, $M$ is discrete and all accumulation points are two-sided. We fix a field $k$, that we assume to be algebraically closed, for simplicity. The corresponding cluster category $\mathcal{C}_{(S,M)}$ is a Hom-finite 2-Calabi-Yau Krull-Schmidt triangulated $k$-category with the following features. In what follows, $[1]$ denotes the suspension functor.
\begin{itemize}
    \item Indecomposable objects, up to isomorphism, are in bijection with the set of all arcs of $(S,M)$. We write $\ell_X$ for the arc corresponding to an indecomposable object $X\in \mathcal{C}_{(S,M)}$.
    \item For $X$ indecomposable with $\ell_X=\{a, b\}$, we have $\ell_{X[1]} = \{\sigma(a), \sigma(b)\}$.
    \item Let $X, Y$ be indecomposable objects. Then $$\rm {Hom}_{\mathcal{C}_{(S,M)}}(X, Y[1]) =  \begin{cases}
    k       & \quad \text{if } \ell_X,\ell_Y \text{ cross}\\
    0       & \quad \text{otherwise} \end{cases}$$
\item Let $\ell_A =\{a_1, a_2\},\ \ell_B=\{a_2, a_3\},\ \ell_C=\{a_3, a_4\},\ \ell_D=\{a_4, a_1\}$ be arcs or boundary segments that are sides of a quadrilateral where $a_1, a_2, a_3, a_4$ are oriented following the orientation of $S$, see Figure~\ref{ex-triangle}. Let $\ell_X=\{a_1, a_3\}$ and $\ell_Y=\{a_2, a_4\}$. Then we have the following non-split exact triangles:
$$X \to B \oplus D \to Y \to X[1]$$
$$Y \to A \oplus C \to X \to Y[1]$$
\end{itemize}
with convention that a boundary segment is identified with a zero object. Therefore, a middle term in one of the above exact triangles might be indecomposable or even trivial.

\begin{figure}[H]
\includegraphics[scale=0.56]{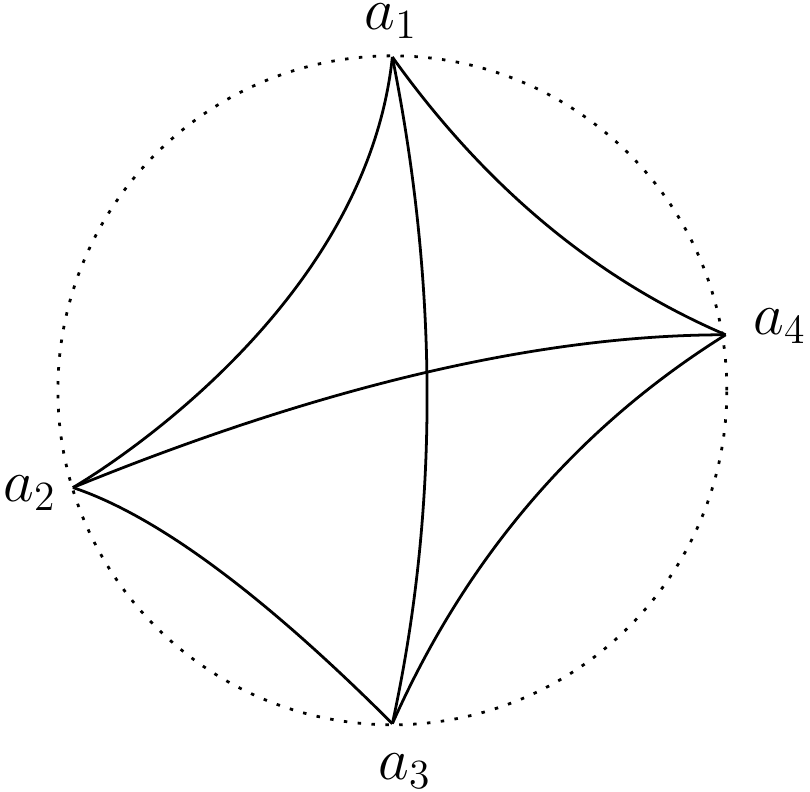}
\caption{Arcs in exact triangles}
\label{ex-triangle}
\end{figure}

\begin{section}{Completion of discrete cluster categories}

We fix $S$ to be the disk and let $M$ be discrete such that the points in $M$ gives rise to finitely many two-sided accumulation points on $\partial S$. We extend the surface $S$ to a new surface $S'$ by replacing each accumulation point with  new marked points ordered as the line of integers. More precisely, we replace each $z_i\in {\rm acc(M)}$ with an interval $[z_i^-,z_i^+]$ containing the points $z_{ij}\in (z_i^-,z_i^+)$ for $j \in \Z$ such that $z_{ij}<z_{ij'}$ if and only if $j<j'$ and $\displaystyle \lim_{j\to \infty} z_{ij}=z_i^+$ and $\displaystyle \lim_{j\to -\infty} z_{ij}=z_i^-$. We set $M' = M \cup \{z_{ij} \mid z_i \in {\rm acc(M)}, j \in \mathbb{Z}\}$. The following is evident.

\begin{proposition} The set $M'$ is discrete in $S'$ and gives rise to finitely many two-sided accumulation points on $\partial S'$. Therefore, we have a discrete cluster category $\mathcal{C}_{(S',M')}$ of type $\mathbb{A}$.
\end{proposition}

Let $\mathcal{D}$ be the full additive subcategory of $\mathcal{C}_{(S',M')}$ generated by the indecomposable objects corresponding to arcs having both their endpoints in the same interval $(z_i^-, z_i^+)$ for some $i$, and let $\mathcal{D}'$ be the full additive subcategory of $\mathcal{C}_{(S',M')}$ generated by the indecomposable objects corresponding to arcs having none of their endpoints in an interval of the form $(z_i^-, z_i^+)$. It is easy to see that if $X$ lies in $\mathcal{D}$ (or in $\mathcal{D}'$), then so does $X[1]$ and $X[-1]$.

Let us recall the definition of perpendicular subcategories. Let $\mathcal{C}$ be a category and $\mathcal{S}$ be a subcategory. We define the following full additive subcategories:

\[\mathcal{S}^{\perp}=\{X\in \mathcal{C}\mid \mathcal{C}(Y,X)=0 \text{ for all } Y\in \mathcal{S}\}\]
\[^{\perp}\mathcal{S}=\{X\in \mathcal{C}\mid \mathcal{C}(X,Y)=0 \text{ for all } Y\in \mathcal{S}\}\]

We have the following which can be easily checked.

\begin{lemma} Let $X$ be an indecomposable object in $\mathcal{C}_{(S',M')}$. Then $X\in \mathcal{D}$ if and only if $Hom_{\mathcal{C}_{(S',M')}}(X,A)=0$ for every $A\in \mathcal{D}'$. Moreover, $\mathcal{D}'=\mathcal{D}^{\perp}= ^{\perp}\mathcal{D}$.
\end{lemma}

In the sequel, when it is clear which category we are working with, we will tend to omit the subscript in the $\Hom$ notation. It will simplify the notations. By a \emph{thick} subcategory of a triangulated category, we mean a triangulated subcategory that is closed under taking direct summands.

\begin{proposition}
The categories $\mathcal{D}$ and $\mathcal{D}'$ are thick subcategories of $\mathcal{C}_{(S',M')}$.
\end{proposition}

\begin{proof}
Observe that for $X$ indecomposable, $X$ lies in $\mathcal{D}$ if and only if ${\rm Hom}(X,D'[1])=0$ for all $D' \in \mathcal{D}'$, and that $X$ lies in $\mathcal{D}'$ if and only if ${\rm Hom}(X,D[1])=0$ for all $D \in \mathcal{D}$. Since the ${\rm Hom}$-functor is additive, $\mathcal{D}$ and $\mathcal{D}'$ are closed under direct summands. We already know that $\mathcal{D}$ and $\mathcal{D}'$ are closed under shifts and inverse shifts. The observation we just had yields that both $\mathcal{D}$ and $\mathcal{D}'$ are closed under taking the cone of a morphism, which proves the statement.
\end{proof}

We will recall some preliminary results about localization of triangulated categories. We will follow \cite{K}. Let $\mathcal{A}$ be a triangulated category and $\Sigma$ be a set of morphisms in $\mathcal{T}$ which is a multiplicative system, i.e. $\Sigma$ admits a calculus of left and right fractions. One can construct a category $\mathcal{A}[\Sigma^{-1}]$, called the \emph{quotient category} of $\mathcal{A}$ by $\Sigma$, such that the morphisms in $\Sigma$ are formally inverted. Let us recall the main ingredients of this construction. First, the objects of $\mathcal{A}[\Sigma^{-1}]$ are the same as the objects of $\mathcal{A}$. For $X$, $Y\in \mathcal{A}$, the pair $(f,g)$ of morphisms $X\xrightarrow{f} Y'\xleftarrow{g} Y$ with $g\in\Sigma$ is called a \emph{left fraction}. Two left fractions $(f_1, g_1) : X\xrightarrow{f_1} Y_1\xleftarrow{g_1} Y$ and $(f_2, g_2): X\xrightarrow{f_2} Y_2\xleftarrow{g_2} Y$ are \emph{equivalent} if there exist an object $Y_3$ and morphisms shown in the commutative diagram
$$\xymatrix{& Y_1 \ar[d] & \\ X \ar[ur]^{f_1} \ar[dr]_{f_2} \ar[r]^{f_3} & Y_3 & Y \ar[ul]_{g_1} \ar[dl]^{g_2} \ar[l]_{g_3} \\ & Y_2 \ar[u] &}$$
with $g_3 \in \Sigma$. As the name suggests, the relation of being equivalent for left fractions is an equivalence relation. A morphism from $X$ to $Y$ in $\mathcal{A}[\Sigma^{-1}]$ is an equivalence class of left fractions. Now, for $(f_1, g_1) : X\xrightarrow{f_1} Y'\xleftarrow{g_1} Y$ and $(f_2, g_2): Y\xrightarrow{f_2} Z'\xleftarrow{g_2} Z$, the composition $(f_2, g_2)\circ(f_1, g_1)$ is obtained by the following diagram
$$\xymatrix{& & W & & \\
& Y' \ar[ur]^u & & Z' \ar[ul]_v &\\
X \ar[ur]^{f_1} && Y \ar[ul]_{g_1} \ar[ur]^{f_2} && Z \ar[ul]_{g_2}}$$
where morphisms $u,v$ are obtained by using the axioms of the multiplicative systems. In other words, the composition is the (well-defined) left fraction $X\xrightarrow{uf_1} W\xleftarrow{vg_2} Z$. When $\mathcal{A}$ is a $k$-category, the quotient category $\mathcal{A}[\Sigma^{-1}]$ is also a $k$-category.
When the multiplicative system $\Sigma$ is compatible with the triangulated structure of $\mathcal{A}$, the quotient category $\mathcal{A}[\Sigma^{-1}]$ also has a unique triangulated structure and the quotient functor $\pi: \mathcal{A} \to \mathcal{A}[\Sigma^{-1}]$ such that for $f: X \to Y$, we have $\pi(f) = (f,\id_Y)$ is an exact functor of triangulated categories~\cite[Lemma 4.3.1]{K}.

In our setting, we declare that a morphism $f: X \to Y$ of $\mathcal{C}_{(S',M')}$ lies in $\Sigma$ if the exact triangle $X \stackrel{f}{\to} Y \to Z \to X[1]$ is such that $Z \in \mathcal{D}$. It follows from Lemma 4.6.1 in \cite{K} that $\Sigma$ is a multiplicative system compatible with the triangulation of $\mathcal{C}_{(S',M')}$. The category $\mathcal{C}_{(S',M')}[\Sigma^{-1}]$ is simply denoted $\overline{\mathcal{C}}_{(S,M)}$. Therefore, we have the following.

\begin{proposition} The quotient category $\overline{\mathcal{C}}_{(S,M)}$ is a triangulated $k$-category and the quotient functor $\pi : \mathcal{C}_{(S',M')} \to \overline{\mathcal{C}}_{(S,M)}$ is an exact functor of triangulated categories.
\end{proposition}

We start with the following easy observation.

\begin{lemma} \label{lemma_iso_fraction}
Let $(f,g): X\xrightarrow{f} Y'\xleftarrow{g} Y$ be an isomorphism in $\overline{\mathcal{C}}_{(S,M)}$. Then $f \in \Sigma$.
\end{lemma}

\begin{proof}
Observe that $(f,g)$ is the composition of $(f,1_{Y'})$ with $(1_{Y'},g)$. Since the latter is an isomorphism, we get that $(f,1_{Y'})$ is an isomorphism. Equivalently, $\pi(f)$ is an isomorphism. However, as $\pi$ is an exact functor, that means $\pi({\rm cone}(f))=0$. Since $\mathcal{D}$ is a thick subcategory, the kernel of $\pi$ is exactly $\mathcal{D}$ and therefore, $f \in \Sigma$.
\end{proof}

To simplify the notations, we simply write $\mathcal{C}$ for $\mathcal{C}_{(S',M')}$ and $\overline{\mathcal{C}}$ for $\overline{\mathcal{C}}_{(S,M)}$. To characterize the indecomposable objects in $\overline{\mathcal{C}}$, we now define an equivalence relation on the set of isoclasses of indecomposable objects in $\mathcal{C}$, or equivalently, on the set of arcs in $\mathcal{C}$ as follows. We say that two arcs $a,b$ are \emph{similar}, written as $a \sim b$, if they become equal when we contract all intervals $[z_i^-, z_i^+]$ back to $z_i$. Arcs that are reduced to one point through this process (this point has to be some accumulation point $z_i$) are precisely the arcs corresponding to indecomposable objects in $\mathcal{D}$. They will be called $\mathcal{D}$-\emph{contractible}. Note that for an indecomposable $D \in \mathcal{D}$, we have $\pi(D)=0$ while $\ell_D$ is $\mathcal{D}$-contractible.

\begin{lemma} \label{lemmatechnical}Let $f: X \to Y$ be a non-zero morphism in $\mathcal{C}$ between indecomposable objects whose arcs are similar. Complete to an exact triangle
\[C \to X \to Y \to C[1]\]
Then either $C \in \mathcal{D}$ or $C$ is a direct sum of objects whose arcs are similar to that of $X$ (and $Y$). In the first case, $f \in \Sigma$, and in the second case, there exists a non-zero morphism $g: Y \to X$ in $\Sigma$. In particular, $X,Y$ are isomorphic in $\overline{\mathcal{C}}$.
\end{lemma}

\begin{proof}
We have $\ell_X = \{a, b\}$ and $\ell_Y=\{a', b'\}$ where $a=a'$ or $a,a'$ belong to the same added interval; and $b=b'$ or $b,b'$ belong to the same added interval.  Consider an exact triangle
$$Y[-1] \to U \oplus V \to X \stackrel{f}{\to} Y.$$ 

If $a = \sigma^i(a')$ for some $i \ge 0$ then both $U, V$ belong to $\mathcal{D}$. In particular, $f$ lies in $\Sigma$. Otherwise, we have $a' = \sigma^i(a)$ for some $i > 0$.  Then none of $U,V$ is the zero object and each of $U, V$ correspond to an arc similar to that of $X$ (and $Y$). In this case, we have an exact triangle
$$X[-1] \to U' \oplus V' \to Y \stackrel{g}{\to} X$$
where both $U', V'$ belong to $\mathcal{D}$. In particular, $g$ lies in $\Sigma$.
\end{proof}

\begin{lemma} \label{lemma1}
Let $f: X \to Y$ be a morphism in $\Sigma$ with $X$  (or $Y$) indecomposable and not in $\mathcal{D}$. Then the other object can be written as $Z \oplus Z'$ where $Z' \in \mathcal{D}$ and $Z$ is indecomposable and not in $\mathcal{D}$ with $\ell_X$ (respectively $\ell_Y$) similar to $\ell_Z$.
\end{lemma}

\begin{proof} Assume that $X$ is indecomposable and not in $\mathcal{D}$. Since $\mathcal{D}$ is a triangulated subcategory, we may discard any direct summand $D$ of $Y$ which is in $\mathcal{D}$, and the resulting morphism $f': X \to Y/D$ will remain in $\Sigma$. That follows from a simple application of the octahedral axiom. Therefore, we may assume that $Y = Y_1 \oplus \cdots \oplus Y_r$ where the $Y_i$ are indecomposable and not in $\mathcal{D}$. We want to prove that $r=1$. Assume that some $\ell_{Y_i}$ is not similar to $\ell_X$. Then we can find $Z \in \mathcal{D}^\perp$ indecomposable such that $\ell_Z, \ell_{Y_i}$ cross but $\ell_Z, \ell_X$ do not. That means that $\Hom(Z,f)$ is not an isomorphism, a contradiction. Therefore, $\ell_{Y_i}$ is similar to $\ell_X$ for all $i$. Now, pick any $Z' \in \mathcal{D}^\perp$ indecomposable such that $\ell_X, \ell_{Z'}$ cross. Observe that $\Hom(Z'[-1],X)$ is one dimensional while $\Hom(Z'[-1],Y)$ is $r$ dimensional, so $r=1$.
\end{proof}

The following lemma is standard and can be easily checked. Compare \cite[Section 2, part (vii)]{GHJ}.

\begin{lemma} \label{FactorizationLemma}
 Let $f: X \to Y$ and $g: Y \to Z$ be morphisms in $\mathcal{C}$ where $X,Y,Z$ are indecomposable such that $gf$ is nonzero. Let 
$\ell_X = \{a,  b\}$ and $\ell_Z= \{c, d\}$ cyclically ordered as $a \le c \le b \le d \le a$ with respect to the orientation of the disk. If $\ell_Y=\{e, f\}$, then $e,f$ are such that $a \le e \le c\le b \le f \le d \le a$ cyclically.
\end{lemma}

\begin{lemma} \label{lemma2}
Let $f: X \to Y$ and $g: Y \to Z$ where $X,Y,Z$ are indecomposable such that $gf \in \Sigma$ is nonzero. Then both $f,g$ are in $\Sigma$.
\end{lemma}

\begin{proof}
Let $\ell_X=\{a, b\}$ and $\ell_Z=\{c, d\}$. Since $gf$ factors through an indecomposable $Y$ with $\ell_Y=\{e, f\}$, we know that $a \le e \le c$ and $b \le f \le d$ by Lemma \ref{FactorizationLemma}. Observe that we have an exact triangle
$$Z[-1] \to E_1 \oplus E_2 \to X \to Z$$
where $\ell_{E_1}=\{a, \sigma^{-1}(c)\}$ and $\ell_{E_2}=\{b, \sigma^{-1}(d)\}$. Therefore, $a, \sigma^{-1}(c)$ belong to the same added interval $I_1$ and $b, \sigma^{-1}(d)$ belong to the same added interval $I_2$. Therefore, all of $a,e,c,\sigma^{-1}(c)$ belong to $I_1$ while all of $b,f,d,\sigma^{-1}(d)$ belong to $I_2$. Therefore, we see that we have an exact triangle
$$Y[-1] \to F_1 \oplus F_2 \to X \stackrel{f}{\to} Y$$
where $F_1, F_2 \in \mathcal{D}$ and
an exact triangle
$$Z[-1] \to G_1 \oplus G_2 \to Y \stackrel{g}{\to} Z$$
where $G_1, G_2 \in \mathcal{D}$.
\end{proof}

\begin{proposition} Let $X$ be an object in $\mathcal{C}$ with no direct summand in $\mathcal{D}$. Then $\pi(X)$ is indecomposable if and only if $X$ is indecomposable. Moreover, if $Y$ is indecomposable in $\mathcal{C}$, then $\pi(X) \cong \pi(Y)$ if and only if $\ell_X, \ell_Y$ are similar.
\end{proposition}

\begin{proof} We know that the kernel of $\pi$ is $\mathcal{D}$, since $\mathcal{D}$ is a thick subcategory,  see~\cite[Proposition 4.6.2]{K}. Therefore, if $\pi(X)$ is indecomposable, then $X$ has to be indecomposable. Conversely, let $X$ be indecomposable. Then there is an indecomposable object $X'$ in $\mathcal{C}$ with $\pi(X')$ indecomposable with a left fraction $X'\xrightarrow{f} Y\xleftarrow{g} X$ that represents a section. Therefore, we have $X\xrightarrow{u} Z\xleftarrow{v} X'$ such that the composition
$$\xymatrix{& & W & & \\
& Y \ar[ur]^p & & Z \ar[ul]_q &\\
X' \ar[ur]^{f} && X \ar[ul]_{g} \ar[ur]^{u} && X' \ar[ul]_{v}}$$
is an isomorphism. We may assume that $Y$ and $Z$ have no direct non-zero summand in $\mathcal{D}$ as otherwise, we may replace $X'\xrightarrow{f} Y\xleftarrow{g} X$ and $X\xrightarrow{u} Z\xleftarrow{v} X'$ respectively by equivalent left fractions $X'\xrightarrow{f'} Y'\xleftarrow{g'} X$ and $X\xrightarrow{u'} Z'\xleftarrow{v'} X'$ where $Y', Z'$ have no non-zero direct summand in $\mathcal{D}$. Since $g, v, qv \in \Sigma$, we get from Lemma \ref{lemma1} that $Y,Z, W$ are indecomposable. Since $pf \in \Sigma$, we know that $f \in \Sigma$, showing that our section is actually an isomorphism. This shows the first part. For the second part, it follows from Lemma \ref{lemmatechnical} that if $\ell_X,\ell_Y$ are similar, then $\pi(X) \cong \pi(Y)$. If $\pi(X) \cong \pi(Y)$, then we have an isomorphism $X\xrightarrow{f} Z\xleftarrow{g} Y$. We know that $Z$ can be chosen to be indecomposable. We know from Lemma \ref{lemma1} that $\ell_X, \ell_Z$ are similar and $\ell_Y, \ell_Z$ are similar. Therefore, $\ell_X, \ell_Y$ are similar.
\end{proof}

\begin{corollary}
 Isoclasses of indecomposable objects in $\overline{\mathcal{C}}$ are in bijection with arcs in $(S, \overline{M})$, where $\overline{M} = M \cup {\rm acc}(M)$ is the closure of $M$.
\end{corollary}

We know that the dimension of Hom-spaces between arcs in $\mathcal{C}$ is at most one. We want to prove the same property for Hom-spaces in $\overline{\mathcal{C}}$ and give a characterization of when such a Hom-space is non-zero. Let $\alpha$ be an equivalence class of arcs. We introduce a partial order $\le_\alpha$ in this equivalence class as follows. We write $\alpha_1 \le_\alpha \alpha_2$ (or simply $\alpha_1 \le \alpha_2$) if there are non-negative integers $i,j$ such that for $\alpha_2=\{a, b\}$, we have $\alpha_1=\{\sigma^i(a), \sigma^j(b)\}$. This is clearly a partial order on the equivalence class of $\alpha$. Observe that if $\alpha_1 \le \alpha_2$ and if $k$ is non-negative and minimal such that $\alpha_1, \alpha_2[k]$ share an endpoint, then $\alpha_2[k]$ is obtained by rotating $\alpha_1$ about the common endpoint by following the orientation (counter-clockwise).

\begin{lemma} Let $X,Y$ be indecomposable objects having similar arcs, and let $f: X\to Y$ be non-zero. If $\ell_X \le \ell_Y$, then $\pi(f)$ is an isomorphism and otherwise, $\pi(f)=0$.
\end{lemma}

\begin{proof} We have an exact triangle $Y[-1] \to E_1 \oplus E_2 \to X \to Y$. The condition $\ell_X \le \ell_Y$ immediately implies that $E_1, E_2$ are in $\mathcal{D}$. If $\ell_X \not\le \ell_Y$, then since $\Hom_\C(X,Y) \ne 0$, this implies that $\ell_Y \le \ell_X$ and $\ell_X, \ell_Y$ cross. But in this case, the above exact triangle splits upon applying the functor $\pi$. Alternatively, let $a, a'$ be the two endpoints of $\ell_X, \ell_Y$, respectively, in a given interval $(z_i^-, z_i^+)$. When $\ell_X \not\le \ell_Y$, the morphism $f$ factors through the object whose arc is $\{a,a'\}$, thus factors through an object in $\mathcal{D}$. Therefore, $\pi(f)=0$.
\end{proof}

\begin{lemma}~\label{lemma_commondenom}
 Let $(f_1,g_1): X\xrightarrow{f_1} Z_1\xleftarrow{g_1} Y$ and $(f_2,g_2): X\xrightarrow{f_2} Z_2\xleftarrow{g_2} Y$ be morphisms in $\overline{\mathcal{C}}$ between indecomposable objects $X,Y$ (which are also indecomposable in $\mathcal{C}$). Then there is an indecomposable object $Z$ in $\overline{\mathcal{C}}$ (which is also indecomposable in $\mathcal{C}$) and morphisms $u_i: Z_i \to Z$, for $i=1,2$, such that $(f_i, g_i)$ is equivalent to $(u_if_i, u_ig_i)$.
\end{lemma}

\begin{proof} If $Z_i$ is not indecomposable, then $Z_i = M_i \oplus D_i$ where $D_i \in \mathcal{D}$. Write $f_i = [f_{i1}, f_{i2}]^T$ and $g = [g_{i1}, g_{i2}]^T$. It is easy to check that each $g_{i1}$ lies in $\Sigma$ and that $(f_i,g_i)$ is equivalent to $(f_{i1},g_{i1})$. For the second part, observe that the arcs of $Z_1, Z_2$ are equivalent. Therefore, we may assume that the $Z_i$ are indecomposable. Since $g_1, g_2 \in \Sigma$, we have $\ell_Y \le \ell_{Z_i}$ for $i=1,2$. Then there is an indecomposable object $Z$ such that $\ell_Z$ is similar to $\ell_Y$ and such that $\ell_{Z_i} \le \ell_Z$ for $i=1,2$. We have morphisms $u_i: Z_i \to Z$ which are in $\Sigma$. Now, each of $(f_i, g_i)$ is equivalent to $(u_if_i, u_ig_i)$. This proves the statement.
\end{proof}

\begin{proposition}~\label{desc-arcs} Let $X,Y$ be indecomposable objects that are not in $\mathcal{D}$. Then $\Hom_{\overline{\mathcal{C}}}(X,Y[1])$ is at most one dimensional. It is one dimensional if and only if one of the following conditions is met for the arcs $\ell_X, \ell_Y$ of $X,Y$ in $(S, \overline{M})$:
\begin{enumerate}[$(1)$]
    \item $\ell_X, \ell_Y$ cross
\item $\ell_X \ne \ell_Y$ share exactly one accumulation point, and we can go from $\ell_X$ to $\ell_Y$ by rotating $\ell_X$ about the common endpoint following the orientation of $S$.
\item $\ell_X = \ell_Y$ have both of their endpoints accumulation points. \end{enumerate}
\end{proposition}

\begin{proof}
It follows from Lemma~\ref{lemma_commondenom} that two morphisms from $X$ to $Y[1]$ in $\overline{\mathcal{C}}$ are given by left fractions $(f_i,g_i): X\xrightarrow{f_i} Z\xleftarrow{g_i} Y[1]$ where $Z$ is indecomposable. Therefore, we may assume that $\Hom_\C(X,Z)$ and $\Hom_\C(Y[1],Z)$ are both non-zero (and hence both one-dimensional). Now, it is straightforward to check that for $\lambda_1, \lambda_2 \in k$ with $\lambda_2 \ne 0$, we have that $(\lambda_1 f,\lambda_2 g)$ is equivalent to $(\lambda_1\lambda_2^{-1} f,g)$. Moreover, for $\alpha_1, \alpha_2 \in k$, we have  $(\alpha_1 f, g) + (\alpha_2 f, g)$ =  $((\alpha_1 + \alpha_2)f,g)$. That proves the first part of the proposition. For the second part, we may assume that $\ell_X, \ell_Y[1]$ are not similar in $(S',M')$, since otherwise, it is clear that $\Hom_{\overline{\mathcal{C}}}(X,Y[1])$ is non-zero if and only if one of the 3 conditions holds. Assume first that $\ell_X,\ell_Y$ cross in $(S,\overline{M})$, and hence in $(S', M')$. Consider the non-split exact triangles

$$Y \to A \oplus B \to X \stackrel{f}{\to} Y[1]$$
in $\C$ where $A$ or $B$ may be the zero object. Then at least one of $A, B$ is not in $\mathcal{D}$. If $\ell_A, \ell_X$ are similar, then since $\ell_A \le \ell_X$, then $\pi$ sends the exact triangle to the split triangle. But in that case, we see that $\ell_X, \ell_Y$ share exactly one point which is an accumulation point in $(S,\overline{M})$ and we can go from $\ell_X$ to $\ell_Y$ by rotating $\ell_X$ about the common endpoint following the opposite orientation of $S$. This contradicts that $\ell_X,\ell_Y$ cross in $(S,\overline{M})$.
Therefore, we may assume that none of $\ell_A, \ell_B$ is similar to $\ell_X$. That means that the above exact triangle does not become split when we apply $\pi$, so that $\pi(f) \ne 0$.
Assume now that $\ell_X,\ell_Y$ do not cross in $(S,\overline{M})$. If $\Hom_{\overline{\mathcal{C}}}(X,Y[1])$ is non-zero, then there is an indecomposable object $Y'$ with $\ell_Y \sim \ell_{Y'}$ such that $\ell_X, \ell_{Y'}$ cross in $(S',M')$. That is only possible if $\ell_X, \ell_Y$ share an accumulation point in $(S, \overline{M})$. We may assume that $\ell_X, \ell_Y[1]$ share exactly one accumulation point in $(S, \overline{M})$, as otherwise, $\ell_X, \ell_{Y[1]}$ would be similar, which has already been treated. As we have seen in the proof of Lemma \ref{lemma_commondenom}, we have that for any arc $\ell_{Y'} \sim \ell_Y$ with $\ell_Y < \ell_{Y'}$, the arcs $\ell_X, \ell_{Y'}$ cross. That implies case (2).  Conversely, assume that case (2) holds. There is a non-split exact triangle 
$$Y[-1] \to A \oplus B \to X \to Y[1]$$
in $\mathcal{C}$ where $A \in \mathcal{D}$ and $B \not \in \mathcal{D}$ and such that $\ell_B$ is neither similar to $\ell_X$, nor to $\ell_Y[-1]$. Therefore, $\pi$ sends this triangle to a non-split triangle, which proves that $\Hom_{\overline{\mathcal{C}}}(X,Y[1])$ is non-zero.
\end{proof}

 Now, let use give a description of the possible non-split exact triangles in $\overline{\mathcal{C}}$ coming from the non-zero extensions between indecomposable objects. Recall the illustration of arcs and exact triangles in Figure~\ref{ex-triangle} of Section~\ref{secondsection}. Arcs configured as in Figure~\ref{ex-triangle} in $\overline{\mathcal{C}}$ give rise to two (non-split) exact triangles in exactly the same way as in the discrete cluster category case. Note that this also works when some of the marked points $a_i$ are accumulations points. That covers case (1) of Proposition \ref{desc-arcs}.

In addition, we have the following configuration of arcs in Figure~\ref{triangleC} which gives rise to a non-split exact triangle \[X\rightarrow B \rightarrow Y\rightarrow X[1].\]
where the open marked point is an accumulation point. That covers case (2) of Proposition \ref{desc-arcs}.

\begin{figure}[H]
\includegraphics[scale=0.42]{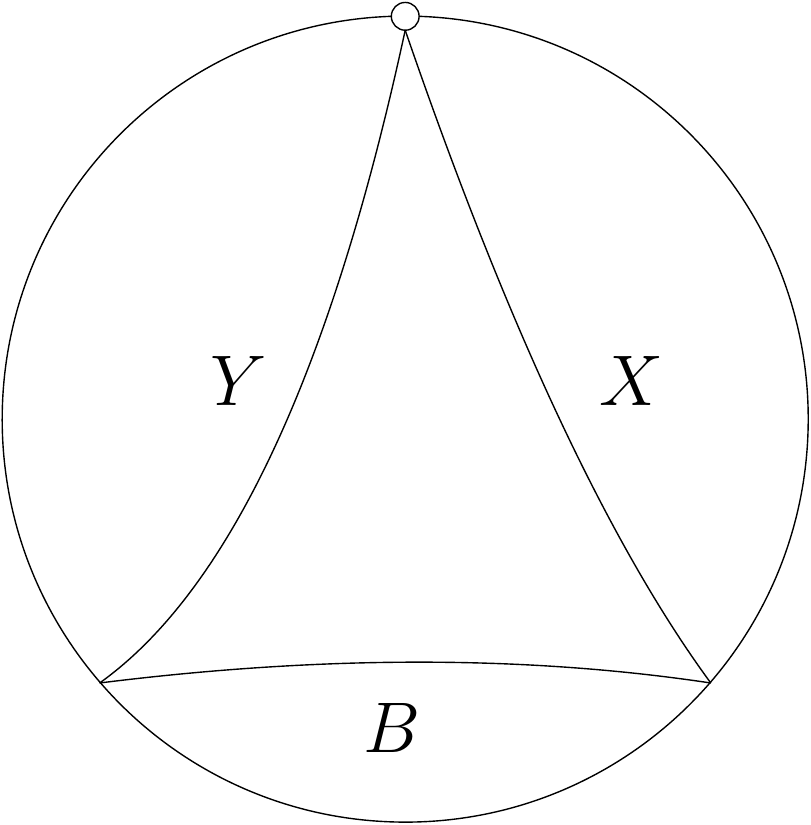}
\caption{Arcs in an exact triangle}
\label{triangleC}
\end{figure}

Finally, if $X$ is such that $\ell_X$ is a limit arcs whose endpoints are both accumulation points, then we have a non-split exact triangle
\[X\rightarrow 0 \rightarrow X\rightarrow X[1]=X.\]
That is the last case of Proposition \ref{desc-arcs}.

\end{section}

\section{Cluster-tilting subcategories}

We let $\mathcal{A}$ denote a Hom-finite triangulated $k$-category and $\mathcal{T}$ be a full additive subcategory of $\mathcal{A}$. All subcategories we consider in this section are full, additive, closed under direct summands and isomorphisms. Let $X$ be an object in $\mathcal{A}$. By $\Hom(\mathcal{T},X)=0$ we mean $\Hom(T,X)=0$ for all $T\in \mathcal{T}$.  Similarly, for $\Hom(X,\mathcal{T})=0$ we mean $\Hom(X,T)=0$ for all $T\in \mathcal{T}$.  By $\mathcal{T}[1]$, we mean the subcategory generated by all $T[1]$ where $T\in\mathcal{T}$ and $[1]$ is the shift (or suspension) functor.  We now give the necessary definitions we will use in this section.

\begin{definition} Let $A$ be an object in $\mathcal{A}$. An object $T\in \mathcal{T}$ together with a morphism $f: T\to A$ is called a \emph{right} $\mathcal{T}$-\emph{approximation} of $A$ if for each $g: T'\to A$ with $T'\in \mathcal{T}$, there exists a morphism $h:T'\to T$ such that $g=f\circ h$. If every object of $\mathcal{A}$ admits a right $\mathcal{T}$-\emph{approximation}, then $\mathcal{T}$ is said to be \emph{contravariantly finite} in $\mathcal{A}$.
\end{definition} 

The definition of a \emph{left} $\mathcal{T}$-\emph{approximation} and covariantly finiteness in $\mathcal{A}$ can be defined dually.

\begin{definition}~\label{defn-cts} Let $\mathcal{A}$ be a $\Hom$-finite triangulated $k$-category. A subcategory $\mathcal{T}$ of $\mathcal{A}$ is \emph{cluster-tilting} if 
\begin{enumerate}[(i)]
\item We have that $\Hom_{\mathcal{A}}(\T,X[1])=0$ if and only if $\Hom_{\mathcal{A}}(X,\T[1])=0$ if and only if $X\in \mathcal{T}$.
\item The subcategory $\mathcal{T}$ is \emph{functorially finite}, i.e. $\mathcal{T}$ is both covariantly finite and contravariantly finite in $\mathcal{A}$.
\end{enumerate}
\end{definition}

Following the notation in~\cite{GHJ}, we define the following types of convergence at an accumulation point.

\begin{definition} Consider $S$ with any set $M$ of marked points. Let $E$ be a given set of arcs of $(S,M)$ and let $z$ be an accumulation point on $\partial S$.

\begin{enumerate}
\item We say that $E$ has a \emph{left fountain} at $z$ if there exists a marked point $m\in M$ with a sequence of arcs $\{m,x_i\}_{i\in\mathbb{Z}_{\geq 0}}$ in $E$, with $x_i \ne z$ for all $i$, such that the marked points $x_i\in M$ converge to $z$ on the left. Here, we call $m$ the \emph{base point} for this left fountain.
\item We say that $E$ has a \emph{right fountain} at $z$ if there exits a marked point $m\in M$ with a sequence of arcs $\{m,x_i\}_{i\in\mathbb{Z}_{\geq 0}}$ in $E$, with $x_i \ne z$ for all $i$, such that the marked points $x_i\in M$ converge to $z$ on the right. We call $m$ the \emph{base point} for this right fountain.
\item We say that $E$ has a \emph{two-sided fountain} at $z$ if it has both a left and a right fountain at $z$ with the same base point $m$.
\item Finally, we say that $E$ has a \emph{leapfrog convergence} at $z$ if there is a sequence of non-crossing arcs, $\{x_i,y_i\}_{i\in\mathbb{Z}_{\geq 0}}$ in $E$, with $x_i \ne z$ and $y_i \ne z$ for all $i$, such that the marked points $x_i\in M$ converge to $z$ on the left and the marked points $y_i\in M$ converges to $z$ on the right.
\end{enumerate}

\end{definition}

\begin{figure}[H]
    \centering
    \includegraphics[scale=0.45]{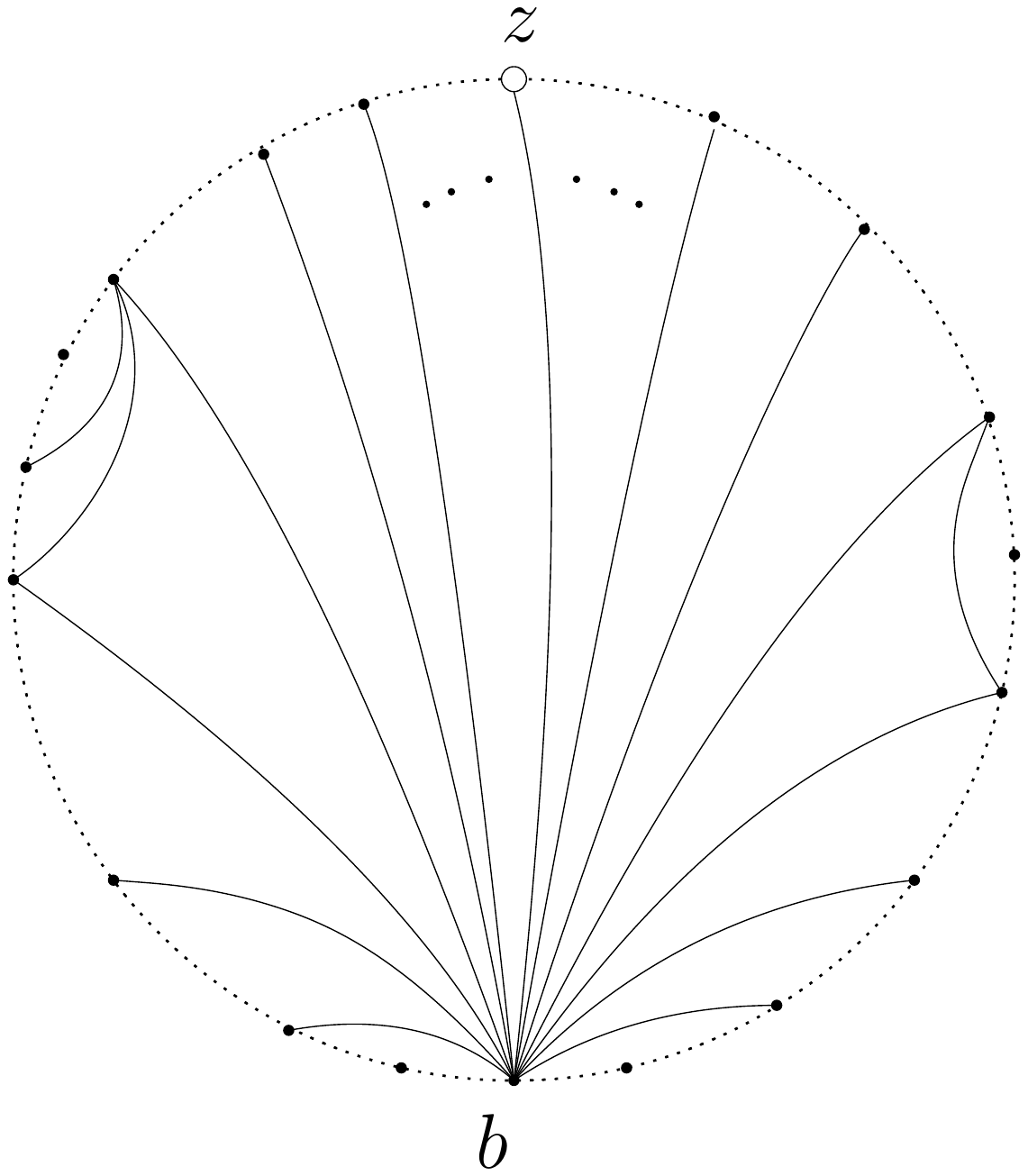}
    \caption{An illustration of a two-sided fountain at an accumulation point with base point $b$.}
    \label{fig:cts4}
\end{figure}

Note that $\mathcal{C}_{(S,M)}$ can be naturally seen as a full triangulated subcategory of $\overline{\mathcal{C}}$. In what follows, a full additive subcategory of $\mathcal{C}_{(S,M)}$ or of $\overline{\mathcal{C}}$ will be identified with its indecomposable objects or, more conveniently, with the corresponding arcs. By a \emph{geometric completion} $\overline{\mathcal{T}}$ of a full additive subcategory $\mathcal{T}$ of $\mathcal{C}_{(S,M)}$, we mean the full additive subcategory $\overline{\mathcal{T}}$ of $\overline{\mathcal{C}}$ generated by the indecomposable objects from $\mathcal{T}$, plus those corresponding to accumulation of arcs (thus limit arcs) of $\mathcal{T}$. Recall the definition of limit arcs in Definition~\ref{lim-arc}. Note that limit arcs in $\mathcal{C}_{(S,M)}$ become normal arcs in $\overline{\mathcal{C}}$, however, we still call them limit arcs. Indeed, these limit arcs have a different behavior, both geometrically and algebraically, as we will see. 

\begin{theorem}~\label{cts} Let $\mathcal{T}$ be a subcategory of $\overline{\mathcal{C}}$. The subcategory $\mathcal{T}$ is cluster-tilting in $\overline{\mathcal{C}}$ if and only if the following conditions hold:
\begin{enumerate}
\item $\mathcal{T} \cap \mathcal{C}_{(S,M)}$ is a maximal collection of non-crossing arcs in $M$.
\item Every accumulation point $z\in acc(M)$ has a two-sided fountain in $\mathcal{T}\cap \mathcal{C}_{(S,M)}$.
\item The subcategory $\mathcal{T}$ is the geometric completion of $\mathcal{T}\cap \mathcal{C}_{(S,M)}$.
\end{enumerate}
\end{theorem}

\begin{remark} The objects in a cluster-tilting subcategory $\mathcal{T}$ come from the non-crossing arcs in $\mathcal{T} \cap \mathcal{C}_{(S,M)}$, plus one limit arc for every two-sided fountain. 
\end{remark}

Before proving Theorem~\ref{cts} we will prove some auxiliary lemmas.

\begin{lemma}~\label{l1} Let $\mathcal{T}$ correspond to a collection of non-crossing arcs of $(S,\overline{M})$ and let $z$ be an accumulation point for which $\mathcal{T}$ does not have a leapfrog convergence at $z$. Then there is a marked point $m$ such that the limit arc $\{z, m\}$ does not cross any arc from $\mathcal{T}$.
\end{lemma}

\begin{proof} If there is an open interval $\mathcal{N}$ containing $z$ such that all but finitely many arcs having an endpoint in $\mathcal{N}$ are entirely contained in $\mathcal{N}$, then since there is no leapfrog convergence at $z$, we can assume further that $\mathcal{N}$ is chosen in such a way that the arcs entirely contained in $\mathcal{N}$ have both of their endpoints on a given side. In that case, we can easily find the wanted arc. Therefore, we may assume that there is a collection $C$ of pairwise distinct arcs whose endpoints on one end are accumulating to $z$, and without loss of generality we can assume they accumulate on the left. Let $\alpha_i = \{z_i,m_i\}$ be those arcs with the $z_i$ accumulating to $z$ on the left. By our assumption above, we can further assume that the $m_i$ do not accumulate to $z$ on the left. 
If the number of $m_i$ is finite, then this means there is a marked point $m$ such that there are infinitely many arcs $\alpha_i$ having $m$ as an endpoint. In this case, the limit arc $\{z, m\}$ does not cross any arc in $\mathcal{T}$. If the number of these $m_i$ is infinite,
then there exists another accumulation point $z'$ such that a subset of the arcs in $C$ accumulate to $z'$. Since these arcs do not cross, the accumulation has to be on the right of $z'$. Take $m=z'$, then the limit arc $\{z, z'\}$ does not cross any arc in $\mathcal{T}$. 
\end{proof}

\begin{lemma}~\label{l0} If there is a limit arc between two accumulation points, then the corresponding object cannot be in a cluster-tilting subcategory.
\end{lemma}

\begin{proof}
Let $\ell_X$ be a limit arc between two accumulation points. We have that $X[1]=X$. This means that $0\neq \Hom(X,X)=\Hom(X,X[1])=\Ext^1(X,X)$. So, $X$ cannot lie in any cluster-tilting subcategory.
\end{proof}

\begin{lemma}~\label{l2} If $\mathcal{T}$ be a cluster-tilting subcategory, then every accumulation point $z$ has exactly one limit arc attached to it.
\end{lemma}

\begin{proof} Assume that there are at least two different arcs $\ell_X$ and $\ell_Y$ attached to an accumulation point $z$ and assume $\ell_Y$ follows $\ell_X$ counter-clockwise, see Figure~\ref{fig:cts1}. Then there is a nonzero homomorphism between $X$ and $Y[1]$ by Proposition~\ref{desc-arcs}. This means that they cannot both lie in the same cluster-tilting subcategory $\mathcal{T}$. 

\begin{figure}[H]
    \centering
    \includegraphics[scale=0.5]{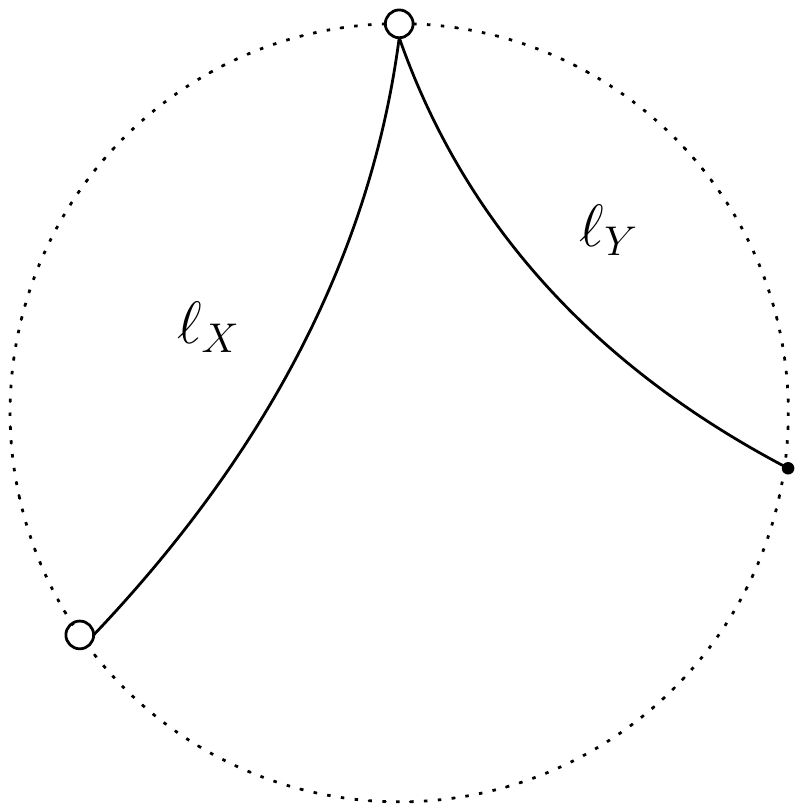}
    \caption{Illustration of two limit arcs.}
    \label{fig:cts1}
\end{figure}

Assume now that there is no arc attached to a given accumulation point $z$. We are going to first show that a cluster-tilting subcategory cannot be formed if there is a leapfrog convergence at $z$. So assume that $\mathcal{T}$ has a leapfrog at $z$. Consider an arc $\ell_M$ for $M\in \overline{\mathcal{C}}$ attached to the accumulation point $z$ such that $\ell_M$ crosses infinitely many other arcs in $\mathcal{T}$. This is possible since we assume there is a leapfrog convergence at $z$. 

We claim that the object $M$ does not have a right $\mathcal{T}$-approximation. Assume the contrary; $T=T^{lf}\oplus T' \to M$ is a right approximation of $M$ where $T^{lf}$ consists of arcs from leapfrog convergence and $T'$ is the complement of $T^{lf}$ in $T$. Since $T$ has a finite number of summands, there is a neighborhood $\mathcal{N}$ around $z$ such that there is no arc from $T$ having an endpoint in $\mathcal{N}$. Take an arbitrary arc $\ell_K[-1]$ from $\mathcal{T}$ such that the endpoints of $\ell_K$ are on different side of $z$ and they are both in $\mathcal{N}$. Note that $\ell_K$ does not cross any arc in $T$, but $\ell_K$ and $\ell_M$ cross. So, this means that there is a map from $K[-1] \to M$ which does not factor through the map $T \to M$.  
This is a contradiction. Thus, a cluster-tilting subcategory cannot come from a collection such that there is a leapfrog convergence at an accumulation point. 

\begin{figure}[H]
    \centering
    \includegraphics[scale=0.5]{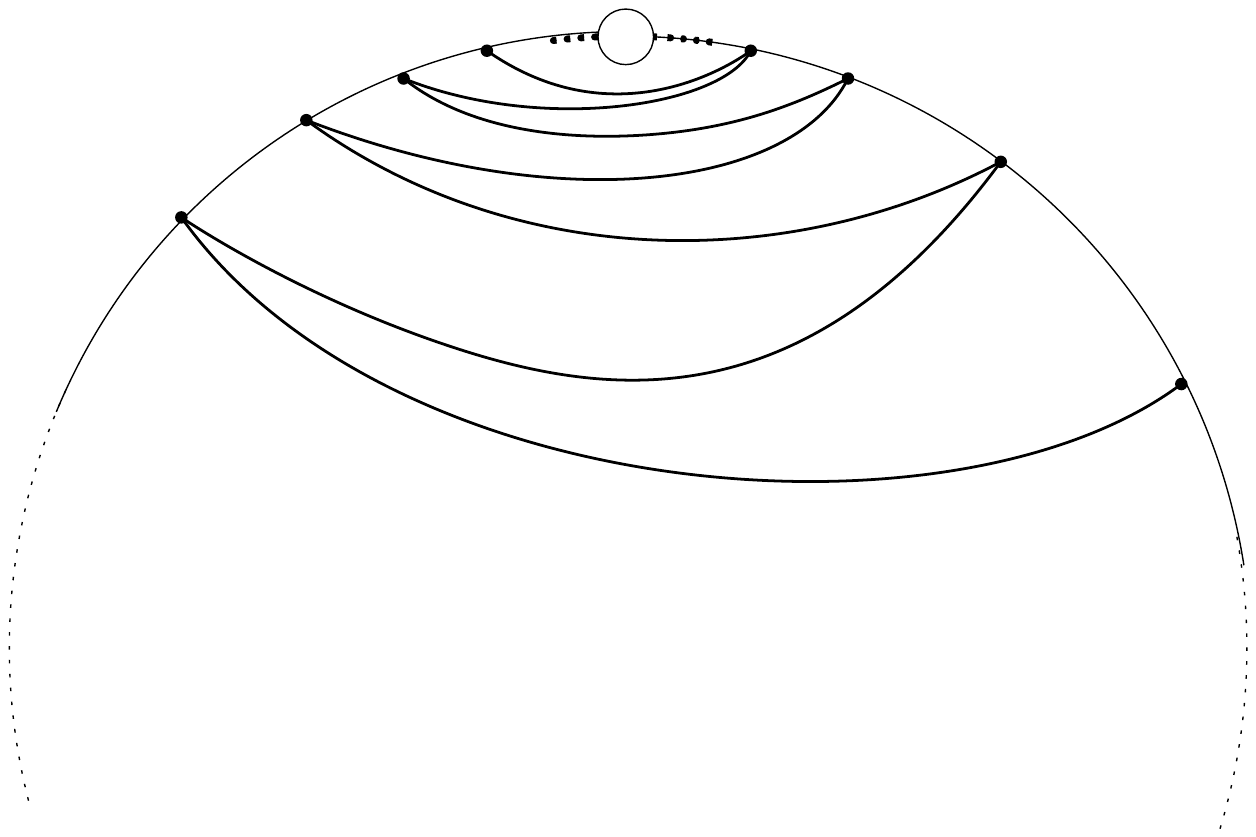}
    \caption{Illustration of a leapfrog convergence at an accumulation point.}
    \label{fig:lf}
\end{figure}

Now we will prove that, in $\mathcal{T}$, if we do not have a leapfrog convergence at an accumulation point $z$ and if there is no arc attached to an accumulation point $z$, then $\mathcal{T}$ cannot be a cluster-tilting subcategory. Lemma~\ref{l1} applies here, i.e. there is a limit arc $\ell_Z=\{z,m\}$ that can be attached to $z$ which will not cross any of the arcs of $\mathcal{T}$.  If $m$ is a regular marked point, that means $\Ext^1(Z,\mathcal{T})=\Ext^1(\mathcal{T},Z)=0$. 
So, by Definition~\ref{defn-cts}, $(i)$, this limit arc should be in $\mathcal{T}$, a contradiction.  If $m$ is an accumulation point and there is no limit arc in $\mathcal{T}$ having $m$ as an endpoint, then the above argument applies to get the same contradiction. If there is (exactly) one limit arc at $m$ in $\mathcal{T}$, then either $\Ext^1(Z,\mathcal{T})= 0$ or $\Ext^1(\mathcal{T},Z)=0$. Again, we get that $Z$ should be in $\mathcal{T}$, a contradiction.
Therefore, we conclude that in a cluster-tilting subcategory, there should be exactly one arc attached to every accumulation point.
\end{proof}

\begin{proof}[Proof of Theorem~\ref{cts}] By Lemma~\ref{l1} and Lemma~\ref{l2} we know that in a cluster-tilting subcategory there is exactly one arc attached to every accumulation point whose other end is a regular marked point. 

Let $z$ be an accumulation point and let $\ell_{X}$ for $X\in \mathcal{T}$ be a limit arc between $z$ and a marked point $b$,  which is a regular marked point.
We claim that $z$ has a two-sided fountain with base $b$. Assume the contrary.
Then either $\mathcal{T}$ does not have a left fountain at $z$ with base $b$, or does not have a right fountain at $z$ with base $b$. We only consider the first case, as the other is treated in a similar way. 

Assume first that $\mathcal{T}$ has a left fountain at $z$ with base $b' \ne b$. Then the arc $\ell_Y = \{z,b'\}$ does not cross the arcs of $\mathcal{T}$ and is not itself in $\mathcal{T}$. See Figure~\ref{prop4} $(i)$. Any right $\mathcal{T}$-approximation $T \to Y$ of $Y$ can only have $X$ and finitely many objects corresponding to arcs from the left fountain at $z$ with base $b'$ as direct summands of $T$. However, if we take an arc $\ell_W$ in that fountain having an endpoint close enough to $z$, then ${\rm Hom}(W,Y) \ne 0$ but ${\rm Hom}(W,T)=0$, a contradiction to the fact that $T \to Y$ is a right $\mathcal{T}$-approximation. Therefore, we may assume that there is no left fountain at $z$ with base point different from $b$. 

Now, assume that some of the arcs having endpoints on the left of $z$ accumulate to an arc $\{z,z'\}$ where $z'$ is another accumulation point. See Figure~\ref{prop4} $(ii)$. Call that infinite family of arcs $F$. In that case, take an arc $\ell_V = \{z',m\}$ where $m$ is on the left of $z$ and close enough to it. Consider the unique limit arc $\ell_U$ of $\mathcal{T}$ having $z'$ as endpoint. A right $\mathcal{T}$-approximation $T \to V$ of $V$ only involves arcs from the family $F$, plus possibly other arcs having both of their endpoints away from $z$. If we take an arc $\ell_W$ in $F$ that is close enough to arc $\{z',z\}$, then ${\rm Hom}(W,V)\ne 0$ while ${\rm Hom}(W,T)=0$, a contradiction. 

Therefore, we are left with the case where any family $\{m_i, n_i\}$ of arcs from $\mathcal{T}$ having the $m_i$ accumulating to $z$ on the left will also have the $n_i$ accumulating to $z$ on the left. In that case, there is an interval $[x,z)$ such that all arcs from $\mathcal{T}$ having an endpoint in $(x,z)$ are entirely in $[x,z)$, and $x$ can be chosen in such a way that there is at least one arc starting at $x$ and ending at a point in $(x,z)$. There is an arc $\{x,x'\}$ in $\mathcal{T}$ where $x'$ lies between $x$ and $z$ and is maximal with respect to this property. 
Similarly, there is an arc $\{x',x''\}$ in $\mathcal{T}$ where $x''$ lies between $x'$ and $z$ and is maximal with respect to this property. Note that $x' \ne x''$ as otherwise, there would be a longer arc than $\{x,x'\}$ starting at $x$ and ending at a marked point in between $x$ and $z$. In that case, we get that $\{x,x''\}$ is an arc of $\mathcal{T}$, a contradiction.

With all of the cases treated, we get that $\mathcal{T}$ has a left fountain at $z$ with base $b$, as wanted. 

\begin{figure}[H]
    \centering
    \includegraphics[scale=0.6]{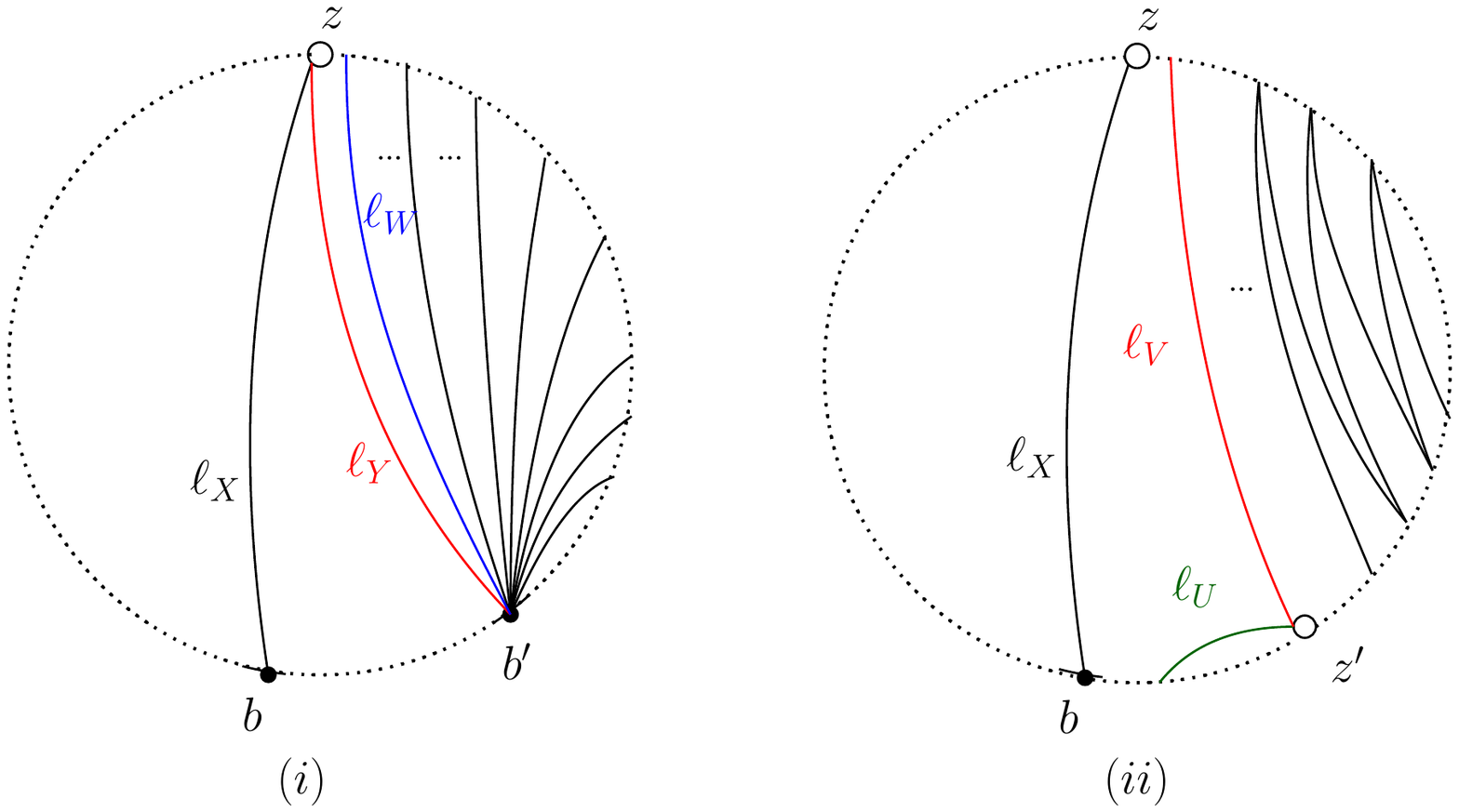}
    \caption{Illustration of the arguments in the Proof of Theorem~\ref{cts}.}
    \label{prop4}
\end{figure}

Finally, we will prove the converse of the statement in the theorem. Assume that we have a subcategory $\mathcal{T}$ which satisfies $(1),(2)$ and $(3)$ in Theorem~\ref{cts}. It is clear that the first condition in the definition of a cluster-tilting subcategory is satisfied. Therefore, we need only to prove that $\mathcal{T}$ is functorially finite. The subcategory $\mathcal{T} \cap \mathcal{C}_{(S,M)}$ is cluster-tilting in $\mathcal{C}_{(S,M)}$ as shown in~\cite{GHJ}. If $\ell_X$ is an ordinary arc, then any $\mathcal{T} \cap \mathcal{C}_{(S,M)}$-approximation of $X$ in $\mathcal{C}_{(S,M)}$ can be completed to a $\mathcal{T}$-approximation of $X$ in $\overline{\mathcal{C}}$ by possibly adding finitely many limit arcs to the approximation thanks to Lemma~\ref{l2}. Remember that we have finitely many accumulation points in our setting and by Lemma~\ref{l2} every accumulation point has one arc attached to it. Therefore, we only need to show that the objects corresponding to limit arcs have both right and left $\mathcal{T}$-approximation.

Take any limit arc $\ell_X$ for $X \in \overline{\mathcal{C}}$ that is not in $\mathcal{T}$. Notice that it crosses infinitely many arcs of $\mathcal{T}$. Let $z$ be an accumulation point such that $\ell_X$ has $z$ as an endpoint. Then $\ell_X$ crosses the arcs of one sided fountain at $z$, and possibly infinitely many arcs of other two-sided fountains. If we discard the arcs in the two-sided fountains in $\mathcal{T}$, we are left with finitely many arcs of $\mathcal{T}$ that have the property of either crossing $\ell_X$ or having a common endpoint with it. Therefore, in order to show that $X$ has $\mathcal{T}$-approximations, it is enough to show that it has $\mathcal{T}'$-approximations where $\mathcal{T}'$ is just the collection of arcs in the two-sided fountains of $\mathcal{T}$ together with its limit arcs. 
Assume that we have $t$ two-sided fountains in $\mathcal{T}$, and let $F_i$ denote the arcs of the $i$-th such fountain. Let $\ell_{A_i}$ denote the limit arc in that fountain, $\ell_{B_i}, \ell_{C_i}$ the furthest arc from the accumulation point in that fountain on the left and on the right, respectively. Note that $\mathcal{T}'$ is the additive closure of the $F_i$ and the $A_i$.  Assume that the fountain at $z$ is $F_1$. We may assume that $\ell_X$ crosses the arcs of $F_1$ that are accumulating to the left of $z$. In this case, any non-zero morphism from an object of $F_1 \cup A_1$ to $X$ has to factor through $A_1$; and any non-zero morphism from $X$ to an object of $F_1 \cup A_1$ has to factor through $B_1$. If $\ell_X$ has another accumulation point $z'$ as an endpoint, then let us call the two-sided fountain at $z'$ by $F_2$. Then $\ell_X$ crosses the arcs of $F_2$ that are accumulating to the left of $z'$. We get a similar analysis for the morphism to or from $X$ and starting from or to $F_2 \cup A_2$. If $\ell_X$ crosses an entire fountain $F_i$, then any non-zero morphism from an object of $F_i \cup A_i$ to $X$ has to factor through $C_i$; and any non-zero morphism from $X$ to an object of $F_i \cup A_i$ has to factor through $B_i$. Doing this for each fountain gives the wanted approximations for $X$.

\begin{figure}[H]
    \centering
    \includegraphics[scale=0.6]{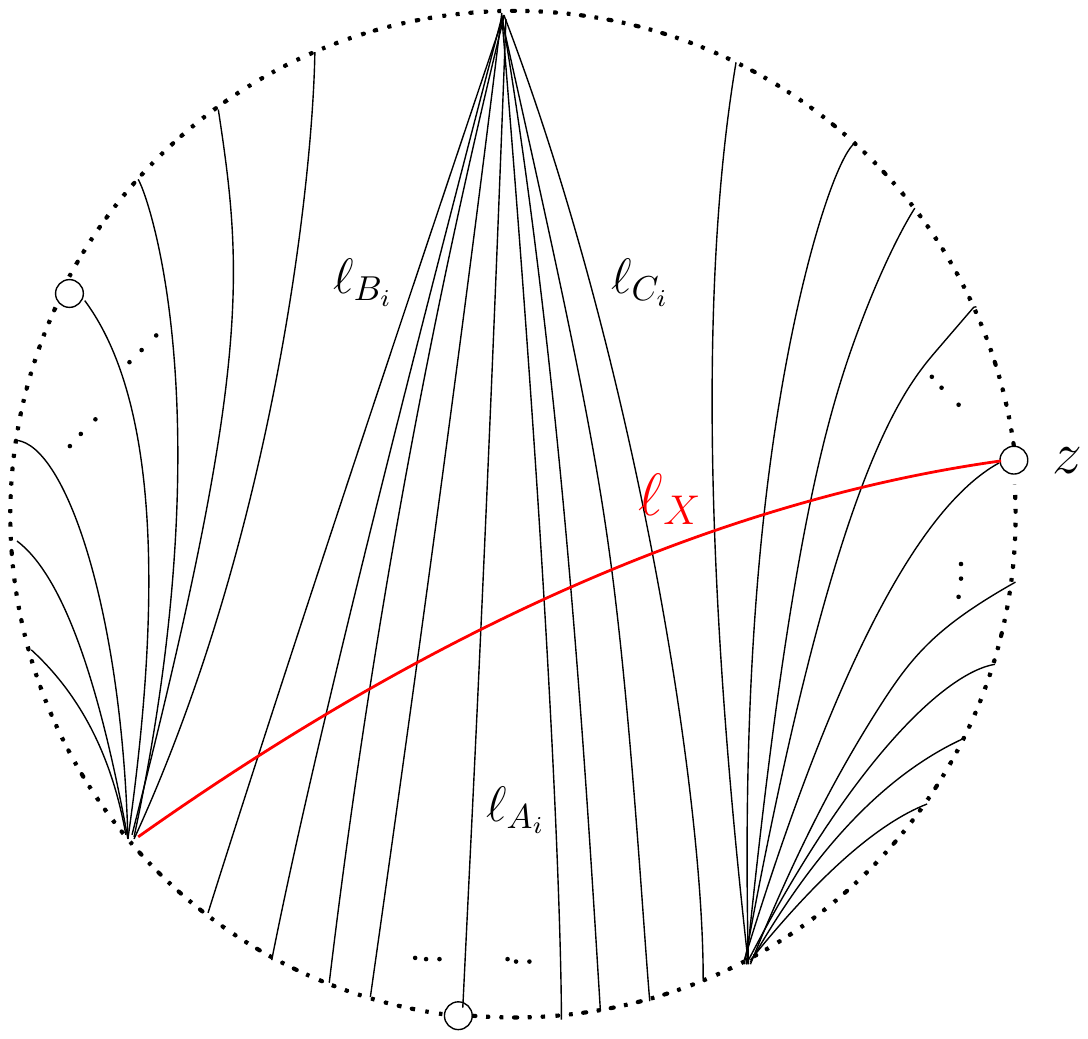}
    \caption{Illustration of the argument in the Proof of Theorem~\ref{cts}.}
    \label{fountains}
\end{figure}
\end{proof}

\section{modules over $\mathcal{T}$}

Let us assume that $\mathcal{A}$ is a Hom-finite Krull-Schmidt triangulated $k$-category, and that $\mathcal{T}$ is a given cluster-tilting subcategory. Given an object $M$ of $\mathcal{A}$, we denote by $\Hom(-, M)|_{\mathcal{T}}$, or simply by $\Hom(-,M)$, the contravariant functor from $\mathcal{T}$ to $\mmod k$ that takes $T \in \mathcal{T}$ to $\Hom_{\mathcal{A}}(T,M)$. We denote by ${\rm Mod}\mathcal{T}$ the category of all contravariant $k$-linear functors from $\mathcal{T}$ to $\mmod k$ and by $\mmod^{\rm fp}\mathcal{T}$ the subcategory of the finitely presented ones. The following theorem can be derived from Corollary 4.4 of~\cite{KZ}.

\begin{theorem} \label{EquivCategories}The functor $\varphi: \mathcal{A} \to {\rm Mod}\mathcal{T}$ that takes $M$ to $\Hom(-,M)$ induces an equivalence
\[\mathcal{A}/{\mathcal{T}[1]}\ \cong \mathrm{mod}^{fp}\ \mathcal{T}\]
\end{theorem}
As a consequence, we get the following, which can also be found in~\cite{KZ}.

\begin{corollary}
The category $\mathrm{mod}^{fp}\mathcal{T}$ is abelian.
\end{corollary}

It is worth noting that it is not generally true that the subcategory of finitely presented functors in a functor category is abelian. We will apply these results in the case where $\mathcal{A} = \overline{\mathcal{C}}$.

\section{Cluster characters}

In this section, we assume that $\mathcal{A}$ is a Hom-finite triangulated Krull-Schmidt $k$-category and that $\mathcal{T}$ is a cluster-tilting subcategory. Cluster characters are well understood in the cases where $\mathcal{T}$ is the additive closure of an object (so corresponds to a cluster-tilting object); see for instance \cite{P,Pl}. Not much work has been done to define cluster characters for general cluster-tilting subcategories. We mention the work of Palu and J\o rgensen \cite{JP}, where the cluster character is defined on a subcategory whose objects correspond to finite dimensional representations through the equivalence of Theorem \ref{EquivCategories}.

We start with some definitions needed to define our cluster character formula. Let $\mathcal{T}_0$ be a complete set of representatives of indecomposable objects of $\mathcal{T}$. For a set $S$ of non-zero morphisms between objects of $V \subseteq \mathcal{T}_0$, consider the quiver $\mathcal{G}(\mathcal{T}, S)$ whose vertex set is $V$ and where there is an arrow $X \to Y$ if and only if there is a morphism from $X$ to $Y$ in $S$. 

Let $F: \mathcal{T} \to {\rm mod}k$ be a $k$-linear (contravariant) functor (also called representation of $\mathcal{T}$, or module over $\mathcal{T}$). A morphism $f: L \to N$ in $\mathcal{T}$ is called an \emph{$F$-isomorphism} if $F(f)$ is an isomorphism. For $F = \Hom_{\mathcal{A}}(-,M)$, we will sometimes use $M$-isomorphism for short. We say that $F$ is \emph{$S$-uniform} if $F(f)$ is an isomorphism for all but finitely many $f \in S$. For a (possibly infinite) quiver $Q$, we say that $Q$ is \emph{finitely co-generated} if there are finitely many vertices $v_1, \ldots, v_r$ such that for any given vertex $v$ of $Q$, there is a finite path $v \rightsquigarrow v_i$ for some $i$. Note that in this case, $Q$ necessarily has finitely many sink vertices.
Finally, we say that the category $\mathcal{T}$ is \emph{quasi-bounded} if there exists a set $S$ of morphisms between objects in $V \subseteq \mathcal{T}_0$ with $\mathcal{G}(\mathcal{T}, S)$ finitely co-generated such that all the representable contravariant functors are $S$-uniform and supported on finitely many vertices in $\mathcal{T}_0 \backslash V$. Although these conditions look technical, we will see that these are the conditions needed to define a cluster characters; see below.

Assume that $\mathcal{T}$ is quasi-bounded for a set of morphisms $S$. Given a representation $F$ of $\mathcal{T}$, its \emph{dimension vector} is the element $({\rm dim}_kF(x))_{x \in \mathcal{T}_0}$ in $(\mathbb{Z}_{\ge 0})^{\mathcal{T}_0}$. Take $M$ in $\mathcal{T}$ and let $\Hom(-,M)$ be the corresponding projective representation. Being $S$-uniform implies that the dimension vector of $\Hom(-,M)$ contains finitely many values, and in particular is bounded.
For $\mathcal{T}$ to be quasi-bounded for a set of morphisms $S$, it is sufficient to check that $\mathcal{G}(\mathcal{T}, S)$ is finitely co-generated and all indecomposable projective representations are $S$-uniform (so the functors $\Hom(-,M)$ where $M$ lies in $\mathcal{T}_0$ are all $S$-uniform) and are supported on finitely many vertices of $\mathcal{T}_0 \backslash V$. This follows from $\mathcal{A}$ being Krull-Schmidt.

\begin{lemma} Let $\mathcal{T}$ be cluster-tilting in $\overline{\mathcal{C}}$. Then there is a set $S$ of morphisms in $\overline{\mathcal{C}}$ that makes $\mathcal{T}$ quasi-bounded.
\end{lemma}

\begin{proof} Let $F_1, \ldots, F_r$ denote the set of all two-sided fountains for $\mathcal{T}$. For each $i$, let $m_i$ denote the base point of the fountain $F_i$, $a_i$ the accumulation point and $\alpha_i$ the accumulation arc from $m_i$ to $a_i$. For each $i$, take a neighborhood $N_i$ of $a_i$ that is not containing any other accumulation point.  For each $i$, take $V_i$ the set of all arcs from $\mathcal{T}$ between $m_i$ and a point in $N_i$. Note that the arcs in $V_i$ are totally ordered, using the orientation of the surface, and there is a minimal and maximal element, and each arc has an immediate successor and immediate predecessor. Take $V$ the union of all $V_i$ together with the limit arcs $\{\alpha_1, \ldots, \alpha_r\}$. Now, the remaining arcs of $\mathcal{T}\backslash V$ are either arcs contained in a region formed by two adjacent arcs from some $V_i$, or finitely many other arcs. Let $W$ be the arcs of the first (possibly infinite) family and let $X$ be the finitely many other arcs. This leads to a partition $\{V, W, X\}$ of $\mathcal{T}$. Observe that any given arc of $(S, \overline{M})$ intersects finitely many arcs from $W$. For our morphism set $S$, we take all irreducible morphisms (up to scaling) in $\mathcal{T}$ between objects from $V$ plus the following additional morphisms. For each fountain $F_i$, take a sequence $\alpha_{ij}$ of arcs in $V$ from $m_i$ to $x_{ij}$ where $\{x_{ij}\}$ converges to $a_i$ on the left. For each $j$, take a corresponding non-zero morphism $f_{ij}: M(\alpha_{ij}) \to M(\alpha_i)$. We need to check that $\mathcal{G}(\mathcal{T}, S)$ is finitely co-generated. For each $V_i$, denote by $\beta_i$ the rightmost arc of $V_i$, that is, the maximal element in $V_i$. Now, for each arc of $V_i$, there is a finite sequence of morphisms from that arc to either $\alpha_i$ or $\beta_i$ (depending on what side that arc is from $a_i$). This proves that the quiver $\mathcal{G}(\mathcal{T}, S)$ is finitely co-generated. Now, one can check that any indecomposable object $M$ of $\overline{\mathcal{C}}$ is such that $\Hom(-,M)$ is $S$-uniform and is supported on finitely many arcs from $W$.
\end{proof}

Given a representation $F$, we let ${\rm supp}_S(F)$ be the subset of $S$ of those morphisms $s$ with $F(s)$ an isomorphism. Observe that $F$ is $S$-uniform if and only if ${\rm supp}_S(F)$ is co-finite in $S$. For the rest of this section, we assume that $\mathcal{T}$ is cluster-tilting and is quasi-bounded for a set of morphisms $S$ between vertices of $V \subseteq \mathcal{T}_0$.  We start with the following lemma.
\begin{lemma} \label{LemmaGeneratingSet}  
Let $F$ be a representation of $\mathcal{T}$ and let $U \subseteq \mathcal{T}_0$ such that for any $v \in \mathcal{T}_0$ with $F(v) \ne 0$, there is a finite path in ${\rm supp}_S(F)$ from $v$ to a vertex in $U$. Then $\oplus_{u \in U}F(u)$ generates $F$.
\end{lemma}

\begin{proof} Let $v$ be a vertex in $\mathcal{T}_0$ with $F(v)$ non-zero. Take a finite path $p$ in ${\rm supp}_S(F)$ from $v$ to $u \in U$. Since $F(p): F(u) \to F(v)$ is an isomorphism, elements in $F(u)$ generate the elements in $F(v)$.
\end{proof}

\begin{lemma}  \label{lemmaUniform}
If $F$ is finitely presented, then $F$ is $S$-uniform and supported on finitely many vertices of $\mathcal{T}_0 \backslash V$. Conversely, if $F$ is $S$-uniform and supported on finitely many vertices of $\mathcal{T}_0 \backslash V$, then $F$ is finitely generated.
\end{lemma}

\begin{proof} Consider a finite presentation
$$P_1 \to P_0 \to F \to 0.$$
Let $S' = {\rm supp}_S(P_0) \cap {\rm supp}_S(P_1)$, which is co-finite in $S$. Let $s: X \to Y \in S'$ and consider the diagram
$$\xymatrix{P_1(Y) \ar[d]^{P_1(s)} \ar[r] & P_0(Y) \ar[d]^{P_0(s)} \ar[r] & F(Y) \ar[d]^{F(s)} \ar[r] & 0 \\
P_1(X) \ar[r] & P_0(X) \ar[r] & F(X) \ar[r] & 0}$$
where the rows are exact and the first two vertical morphisms are isomorphisms. Diagram chasing shows that $F(s)$ is an isomorphism. Therefore, $S' \subseteq {\rm supp}_S(F)$, showing that $F$ is $S$-uniform. The fact that $F$ is supported on finitely many vertices of $\mathcal{T}_0 \backslash V$ follows from the same property for $P_0$. Conversely, assume that $F$ is $S$-uniform and supported on finitely many vertices of $\mathcal{T}_0 \backslash V$. Notice that ${\rm supp}_S(F)$ is co-finite in $S$. Since $\mathcal{G}(\mathcal{T}, S)$ is finitely co-generated, there are finitely many vertices $v_1, \ldots, v_r$ in $\mathcal{G}(\mathcal{T}, S)$ such that for any given vertex $v$, there is a finite path from $v$ to one of these vertices. Let $S''$ be the set of initial vertices of the arrows in $S \backslash {\rm supp}_S(F)$. Then the finite set $U:=\{v_1, \ldots, v_r\}\cup S''$ of vertices is such that for any vertex $v$ of $\mathcal{G}(\mathcal{T},S)$, there is a finite path in ${\rm supp}_S(F)$ from $v$ to a vertex in $U$. Take $U'$ the finite set of vertices in $\mathcal{T}_0 \backslash V$ supporting $F$. Now,  it follows from Lemma \ref{LemmaGeneratingSet} that $\oplus_{v \in U \cup U'}F(v)$ generates $F$.
\end{proof}

Note that if $F$ is finitely presented of dimension vector $f$, then an $f$-dimensional representation need not be finitely presented. However, we have the following.

\begin{lemma} \label{lemmafinitelypresented} Let $F$ be finitely presented with dimension vector $f$. Let $g$ be the dimension vector of a finitely presented sub-representation $G$ of $F$. Then
\begin{enumerate}[$(1)$]
    \item All sub-representations of $F$ of dimension vector $g$ are finitely presented.
\item There are finitely many objects $M_1, \ldots, M_t$ in $\mathcal{T}_0$ such that for any sub-representation $G'$ of dimension vector $g$, $G'$ is generated by $G(M)$ where $M = M_1 \oplus \cdots \oplus M_t$.
\item There are finitely generated projective modules $P_0, P_1$ such that for any sub-representation $G'$ of dimension vector $g$, $G'$ admits a projective presentation $P_1 \to P_0 \to G'$. \end{enumerate}
\end{lemma}

\begin{proof} Let $G$ be a subrepresentation of $F$ with dimension vector $g$ that is finitely presented. Let $G'$ be a subrepresentation of $F$ of dimension vector $g$. Consider the short exact sequence
$$0 \to G' \to F \to F/G' \to 0.$$
Assume for the moment that $G'$ is finitely generated. Thus, there is an epimorphism $f: P_1 \to G'$ where $P_1$ is finitely generated projective. Since $F$ is finitely presented, there is also an epimorphism $P_2 \to F$ where $P_2$ is finitely generated projective. This yields an epimorphism $g: P_2 \to F/G'$. Therefore, we have a commutative diagram
$$\xymatrix{0 \ar[r] & P_1 \ar[d]^f \ar[r] & P_1 \oplus P_2 \ar[d]^h \ar[r] & P_2 \ar[r] \ar[d]^g & 0 \\
0 \ar[r] & G' \ar[r] & F \ar[r] & F/G' \ar[r] & 0}$$
with exact rows where $h$ is the induced morphism. This yields an epimorphism ${\rm ker}(h) \to {\rm ker}(g)$. Since $F$ is finitely presented, ${\rm ker}(h)$ is finitely generated, so ${\rm ker}(g)$ is also finitely generated. This proves that $F/G'$ is finitely presented. Since ${\rm mod}^{\rm fp} \mathcal{T}$ is abelian, $G'$ is finitely presented. 

Therefore, in order to prove that $G'$ is finitely presented, we only need to prove that $G'$ is finitely generated. By Lemma \ref{lemmaUniform}, it is sufficient to prove it is $S$-uniform and supported on finitely many vertices of $\mathcal{T}_0 \backslash V$. To check the latter condition, let $W$ be the finitely many vertices of $\mathcal{T}_0 \backslash V$ that are supporting $F$. Then $W$ contains the vertices of $\mathcal{T}_0 \backslash V$ that are supporting $G'$. Let $S' = {\rm supp}_S(F) \cap {\rm supp}_S(G)$, which is co-finite in $S$.
Observe that for $s \in S'$, $G'(s)$ is the restriction of $F(s)$. Since the latter is an isomorphism, $G'(s)$ is a monomorphism. Let $s:x \to y$. Using that $s \in {\rm supp}_S(G)$, we get that $G(s)$ is an isomorphism, so $G(x), G(y)$ have the same dimension. Since $G,G'$ have the same dimension vector, we obtain that $G'(s)$ is indeed an isomorphism. Therefore $G'$ is $S$-uniform, which completes the proof of the first part. 

For the second part, observe that if $G, G'$ both have dimension vector $g$ and are subrepresentations of $F$, then ${\rm supp}_S(F) \cap {\rm supp}_S(G) = {\rm supp}_S(F) \cap {\rm supp}_S(G')$ is co-finite in $S$. Since $\mathcal{G}(\mathcal{T},S)$ is finitely co-generated, there are finitely many vertices $v_1, \ldots, v_r$ such that for any given vertex $v$ of $\mathcal{G}(\mathcal{T},S)$, there is a finite path from $v$ to one of these vertices. Let $S'$ be the set of initial vertices of the arrows in $S \backslash ({\rm supp}_S(F) \cap {\rm supp}_S(G))$. Then the set $U:=\{v_1, \ldots, v_r\}\cup S'\cup W$ of vertices is such that $\oplus_{v \in U}G'(v)$ generates $G'$,  by Lemma \ref{LemmaGeneratingSet}. This yields (2).

 In order to prove $(3)$, let $G,G'$ be two subrepresentations of dimension vector $g$, which we know are finitely presented. It follows from $(2)$ that there are epimorphisms $u_1: P_0 \to G$ and $u_2: P_0 \to G'$ where $P_0$ is finitely generated projective. Therefore, both ${\rm ker}u_1, {\rm ker}u_2$ are finitely generated subrepresentations of a finitely presented representation $P_0$. Since we know that the category of finitely presented representations is abelian, this gives that ${\rm ker}u_1, {\rm ker}u_2$ are actually finitely presented. Therefore, we can apply Statement $(2)$ in this setting, and there are epimorphisms $v_1: P_1 \to {\rm ker}u_1$ and $v_2: P_1 \to {\rm ker}u_2$ where $P_1$ is finitely generated projective.
\end{proof}
 
\begin{remark}
Lemma \ref{lemmafinitelypresented} is what we need to define cluster characters. We could have started this section with a triangulated category $\mathcal{A}$ together with a cluster-tilting subcategory $\mathcal{T}$ that satisfy Lemma \ref{lemmafinitelypresented}. Our notion of being quasi-bounded for a set $S$ implies the properties stated in that lemma.
\end{remark}

Now, given a finitely presented representation $F$ and a dimension vector $g$, one says that $g$ is \emph{finitely presented} in $F$ if there is a finitely presented subrepresentation of $F$ of dimension vector $g$. By Lemma \ref{lemmafinitelypresented}, in this case, all subrepresentations of $F$ of dimension vector $g$ are finitely presented. Given a finitely presented representation $F$ and a finitely presented dimension vector $g$ in $F$, recall from Lemma \ref{lemmafinitelypresented} that there are finitely generated projective representations $P_0$ and $P_1$ with the property that any subrepresentation of $F$ of dimension vector $g$ has a projective presentation of shape $P_1 \to P_0$. We fix an object $M(F,g)$ such that $P_0 \oplus P_1$ lies in ${\rm add}\Hom(-,M(F,g))$.  In particular, for any subrepresentation $G$ of $F$ of dimension vector $g$, we have that $G(M(F,g))$ generates $G$. We define the quiver $Q(F,g)$ as being the Gabriel quiver of $A(F,g):={\rm End}(M(F,g))^{\rm op}$. There is a restriction functor $\psi(Q,g): \mmod^{\rm fp}\mathcal{T} \to \mmod A(F,g)$ that takes $G$ to the restriction of $G$ on the additive subcategory of $\mathcal{T}$ generated by $M(F,g)$. We define the Grassmannian ${\rm Gr}_gF$ to be the classical Grassmannian of subrepresentations of $\psi(Q,g)F$ of dimension vector $\psi(Q,g)g$. {\sl A priori}, this definition depends on the chosen object $M(F,g)$, however, the lemma below implies that we can add finitely many summands to $M(F,g)$, and this will not change the Grassmannian.

\begin{lemma} \label{LemmaGrass}Using the above notation, let $M$ be an object of $\mathcal{T}$ having $M(F,g)$ as a direct summand and let $A'$ be the opposite endomorphism algebra of $M$. Let $\psi': \mmod^{\rm fp}\mathcal{T} \to \mmod A'$ that takes $H$ to the restriction of $H$ on the additive subcategory of $\mathcal{T}$ generated by $M$. Then ${\rm Gr}_gF$ is isomorphic to ${\rm Gr}_{g'}F'$ as projective varieties, where $F' = \psi'(F)$ and $g' = \psi'(g)$.
\end{lemma}

\begin{proof}
Let $\rho: \mmod A' \to \mmod A(F,g)$ be the restriction functor which coincides with $H \mapsto \Hom_{A'}(\Hom(M,M(F,g)),H)$ for a right $A'$-module $H$.
Consider the left adjoint functor $\lambda: \mmod A(F,g) \to \mmod A'$ given by $H' \mapsto H'\otimes_{A(F,g)} \Hom(M,M(F,g))$ for a right $A(F,g)$-module $H'$. Observe that $\rho, \gamma$ are equivalences when one restricts to finitely generated projective right $A'$-modules in ${\rm add}\Hom(M,M(F,g))$ and all finitely generated projective right $A(F,g)$-modules. Observe also that the minimal projective presentation of any finitely presented $g$-dimensional subrepresentation $G$ of $F$ has terms lying in ${\rm add}\Hom(-,M(F,g))$. Therefore, it follows that $\lambda\rho\psi'(G) \cong \psi'(G)$ and that $\rho\lambda\psi(G) \cong \psi(G)$. Therefore, $\rho$ gives rise to a bijection between finitely presented subrepresentations of $\psi'(F)$ of dimension vector $\psi'(g)$ to finitely presented subrepresentations of $\psi(F)$ of dimension vector $\psi(g)$. This morphism is a projection and hence a morphism of projective varieties.
\end{proof}

\subsection{Indices and Coindices}

Consider a set $\{x_i\}_{i \in \mathcal{T}_0}$ of indeterminates. For $i \in \mathcal{T}_0$, we denote by $U_i$ the corresponding indecomposable object of $\mathcal{T}$. Let ${\rm Mon}$ be the set of all finite Laurent monomials. Consider $B = \prod_{m \in {\rm Mon}}\mathbb{Z}$. An element in $B$ can be thought of as an infinite linear combination of the Laurent monomials. Note that although some elements in $B$ can be multiplied naturally, in the sense of power series multiplication, the multiplication is not always defined. For instance, $(1/x_i)_{i \in \mathcal{T}_0}$ cannot be multiplied with $(x_i)_{i \in \mathcal{T}_0}$. We will eventually define a subgroup of $B$ that will yield a $\mathbb{Z}$-algebra. Let $K_0(\mathcal{T})$ denote the Grothendieck group of $\mmod^{\rm fp}\mathcal{T}$.  For a finitely presented representation $F$, we denote by $[F]$ the corresponding element in $K_0(\mathcal{T})$ which can be identified with the dimension vector of $F$. Observe that $K_0(\mathcal{T})$ is the subgroup of $\prod_{i\in \mathcal{T}_0}\mathbb{Z}$ consisting of the dimension vectors of the finitely presented representations. It is worth noting that in general, $K_0(\mathcal{T}) \ne \prod_{i\in \mathcal{T}_0}\mathbb{Z}$. Also, the dimension vectors of the finitely generated indecomposable projectives in $K_0(\mathcal{T})$ may not generate $K_0(\mathcal{T})$. We will actually work with the subgroup $K_0'(\mathcal{T})$ of $K_0(\mathcal{T})$ generated by the dimension vectors of the projectives. For an object $M \in \mathcal{A}$, we consider a triangle
$$T_1 \to T_0 \to M \to T_1[1]$$
where $T_0 \to M$ is a minimal right $\mathcal{T}$-approximation of $M$ and, thanks to $\mathcal{T}$ being cluster-tilting, $T_1 \in \mathcal{T}$ as well. We define the \emph{index} of $M$ to be ${\rm ind}(M)=[\varphi T_0] - [\varphi T_1] \in K_0'(\mathcal{T})$, where we recall that $\varphi$ is the functor in Theorem \ref{EquivCategories}. Note that there is a similar triangle
$$M \to T_0'[2] \to T_1'[2] \to M[1]$$
where $T_0', T_1' \in \mathcal{T}$. We define the \emph{coindex} of $M$ to be ${\rm coind}(M)=[\varphi T_0'] - [\varphi T_1']$. The arguments given in \cite{P} carry to our setting to yield the following two lemmas.

\begin{lemma} The index and coindex of an object in $\mathcal{A}$ are well defined in the sense that they do not depend on the chosen exact triangles above.
\end{lemma}

\begin{proof}
As explained in \cite[Lemma 2.1]{P}, if one chooses another right $\mathcal{T}$-approximation $T_0' \to M$ of $M$, then we get a corresponding exact triangle
$$T_1' \to T_0' \to M \to T_1[1]$$
which is isomorphic to the direct sum of the above exact triangle $$T_1 \to T_0 \to M \to T_1[1]$$
with a split triangle $Z \stackrel{1}{\to} Z \to 0 \to Z[1]$ where $T_0' \cong T_0 \oplus Z$ and $T_1' = T_1 \oplus Z$.
\end{proof}

More importantly, as shown in \cite{P}, the function ${\rm coind} - {\rm ind}: {\rm obj}(\mathcal{A}) \to K_0'(\mathcal{T})$ induces a well-defined function ${\rm coind} - {\rm ind}: K_0(\mathcal{T}) \to K_0'(\mathcal{T})$.

\begin{lemma}
If $F_1, F_2$ are finitely presented with the same dimension vector, then ${\rm coind}(F_1) - {\rm ind}(F_1) = {\rm coind}(F_2) - {\rm ind}(F_2)$.
\end{lemma}

\begin{proof}
The proof essentially follows the same arguments as given in \cite{P}. First, let
$$0 \to X \to Y \to Z \to 0$$
be a short exact sequence in $\mathrm{mod}^{fp}\mathcal{T}$. Then following exactly the same arguments as in Lemma 3.1 of \cite{P}, there exists an exact triangle
$$\eta: X' \to Y' \to Z' \stackrel{\epsilon}{\to} X'[1]$$
in $\mathcal{A}$ such that $\varphi$ maps $\eta$ to the above short exact sequence.  The next step is to check that ${\rm coind}(U) - {\rm ind}(U)$ depends only on $\varphi(U)$. This is Lemma 2.1(4) of \cite{P} and again, the same argument works in our setting. To complete the proof, we need to check that
$${\rm coind}(Y') - {\rm ind}(Y') = {\rm coind}(X') - {\rm ind}(X') + {\rm coind}(Z') - {\rm ind}(Z').$$
Since $\epsilon$ factors through an object in $\mathcal{T}[1]$, it follows from the argument given in the first part of the proof of Proposition 2.2 in \cite{P} that ${\rm ind}(Y') = {\rm ind}(X') + {\rm ind}(Z')$. Now, the triangle
$$X'[-1] \to Y'[-1] \to Z'[-1] \stackrel{\epsilon'}{\to} X'$$
is such that $\epsilon'$ factors through an object in $\mathcal{T}[1]$. Therefore, we have
${\rm ind}(Y'[-1]) = {\rm ind}(X'[-1]) + {\rm ind}(Z'[-1])$
which yields $-{\rm coind}(Y') = -{\rm coind}(X') - {\rm coind}(Z')$.

\end{proof}

\subsection{Cluster Characters}
Note that the coindex and index of an object are always finite linear combinations of the dimension vectors of the indecomposable projectives. For a finitely presented representation $F$, let us write
$${\rm coind}(F) = \sum_{i \in \mathcal{T}_0}a_i[\varphi U_i]$$
$${\rm ind}(F) = \sum_{i \in \mathcal{T}_0}b_i[\varphi U_i]$$
where both sums are finite sums. We denote by $X^{{\rm coind}(F)}$ (resp. $X^{{\rm ind}(F)}$) the Laurent monomial such that for $i \in \mathcal{T}_0$, the exponent of $x_i$ is  $a_i$ (resp. $b_i$). By the above lemma, for $g \in K_0(\mathcal{T})$, we write $X^{{\rm coind}(g) - {\rm ind}(g)}$ for the Laurent monomial $X^{{\rm coind}(G)}X^{-{\rm ind}(G)}$ where $G$ is any finitely presented representation of dimension vector $g$. We define the map
$$X^{\mathcal{T}}: {\rm obj}(\mathcal{A}) \to B$$
with the formula
$$X^{\mathcal{T}}(M) = X^{-{\rm coind}(\varphi(M))}\prod_g \chi({\rm Gr}_g(\varphi(M)))X^{{\rm coind}(g) - {\rm ind}(g)}$$
where the product runs through the finitely presented dimension vectors in $\varphi(M)$. Recall that if $\mathcal{A}$ is in addition $2$-Calabi-Yau, then a function $\zeta:{\rm obj}(\mathcal{A}) \to B$ is a \emph{cluster character} if $\zeta(M \oplus N) = \zeta(M)\zeta(N)$ for all objects $M,N$; and if $M,N$ are indecomposable with $\Hom(M,N [1])$ one dimensional, then $\zeta(M)\zeta(N) = \zeta(B_1)+\zeta(B_2)$ where we have two non-split exact triangles
$$M \to B_1 \to N \to M[1]$$
$$N \to B_2 \to M \to N[1].$$ The first property is called the \emph{multiplication formula}, while the second is called the \emph{exchange formula}. Since we are working with a category $\mathcal{A}$ that does not always have a Serre functor, we will have to restrict the exchange formula to objects in a full subcategory of $\mathcal{A}$ that has a Serre functor. This will be defined in Theorem \ref{ExchangePropertyCY}. We start with the multiplication formula.

\begin{proposition}
We have $X^{\mathcal{T}}(N_1 \oplus N_2) = X^{\mathcal{T}}(N_1)X^{\mathcal{T}}(N_2)$.
\end{proposition}
\begin{proof} Let $\varphi(N_i)=F_i$ and $\varphi(N_1 \oplus N_2)=F$ so that $F \cong F_1 \oplus F_2$. Observe that if $g_i$ is a finitely presented dimension vector in $F_i$, then $g_1 + g_2$ is a finitely presented dimension vector in $F$. Moreover, if $g= g_1 + g_2$ is finitely presented in $F$ with $g_1$ (or $g_2$) finitely presented in $F_1$ (resp. $F_2$), then $g_2$ (resp. $g_1$) is either finitely presented in $F_2$ (resp. $F_1$) or there is no sub-representation of $F_2$ (resp. of $F_1$) of that dimension vector. We consider $M: = M(F,g)$ and note that if $g=g_1 + g_2$ with $g_i$ finitely presented in $F_i$, then $M(F_i,g_i)$ is a direct summand of $M$. In particular, we have finitely many possible decompositions $g=g_1 + g_2$ where $g_1,g_2$ are finitely presented in $F_1, F_2$, respectively. We consider the restriction functor $\psi: {\rm \mmod}^{\rm fp}\mathcal{T} \to \mmod {\rm End}(M)^{\rm op}$. The classical formula
$$\chi({\rm Gr}_g(\psi F)) = \sum_{g=g_1+g_2}\chi({\rm Gr}_{g_1}(\psi F_1))\chi({\rm Gr}_{g_2}(\psi F_2))$$
on the Euler characteristic of Grassmannians translates to
$$\chi({\rm Gr}_g(F)) = \sum_{g=g_1+g_2}\chi({\rm Gr}_{g_1}(F_1))\chi({\rm Gr}_{g_2}(F_2)),$$
by using Lemma \ref{LemmaGrass} and the definition of Grassmannian for representations of $\mathcal{T}$. Now, the result follows from using additivity of the coindex and index.
\end{proof}

\begin{lemma} Assume that $\mathcal{A}$ has a Serre functor $\mathbb{S}$ and that $\mathcal{A}'$ is the Verdier quotient of $\mathcal{A}$ by a thick subcategory $\mathcal{S}$. Assume further that $\mathbb{S}$ is stable on $\mathcal{S}$. Then $\mathbb{S}$ induces a functor $\overline{\mathbb{S}}: \mathcal{A}' \to \mathcal{A}'$ such that $\overline{\mathbb{S}} \pi = \pi \mathbb{S}$ where $\pi: \mathcal{A} \to \mathcal{A'}$ is the quotient functor.
\end{lemma}

\begin{proof}
Since $\mathbb{S}$ is an exact functor that is stable on the subcategory $\mathcal{S}$, the composition $\pi \mathbb{S}$ annihilates $\mathcal{S}$ and therefore, from Proposition 4.6.2 in \cite{K}, it follows that there is an exact functor $\overline{\mathbb{S}}: \mathcal{A}' \to \mathcal{A}'$ such that $\overline{\mathbb{S}} \pi = \pi \mathbb{S}$.
\end{proof}

We apply this for our categories $\mathcal{C} = \mathcal{C}_{(S',M')}$ and $\overline{\mathcal{C}}$.  Recall that $\mathcal{C}$ has a Serre functor which coincides with the second power $[2]$ of the shift functor. The next lemma shows that even though $\overline{\mathcal{C}}$ does not have a Serre functor, the functor  $\overline{\mathbb{S}}$ behaves like a Serre functor for the objects in 
the subcategory $\mathcal{C}_{(S,M)}$. Note that the functorial isomorphism below is at the level of the entire $\overline{\mathcal{C}}$. The restriction of $\overline{\mathbb{S}}$ to the subcategory $\mathcal{C}_{(S,M)}$ is the Serre functor of this subcategory.

\begin{lemma} \label{locallyCY}
Let $M$ be an indecomposable object in $\overline{\mathcal{C}}$ corresponding to an ordinary arc. Then there is a functorial isomorphism
$$\Hom_{\overline{\mathcal{C}}}(M,-) \cong D\Hom_{\overline{\mathcal{C}}}(-,\overline{\mathbb{S}}M) \cong D\Hom_{\overline{\mathcal{C}}}(-,M[2]).$$
\end{lemma}

\begin{proof}
We have a non-split exact triangle
$$\eta: \quad \mathbb{S}M[-1] \stackrel{f}{\to} E \stackrel{g}{\to} M \stackrel{h}{\to} \mathbb{S}M$$
in $\mathcal{C}$ where any non-retraction $g': Z \to M$ is such that $hg'=0$ and $\mathbb{S}M$ is indecomposable. If $\pi(h)=0$, then there is a direct summand $E'$ of $E$ such that $\pi$ sends $E' \to M$ to an isomorphism. That means that the arc of $E'$ is equivalent to that of $M$. Since $M$ corresponds to an ordinary arc, that means that $E'$ is isomorphic to $M$ in $\mathcal{C}$ and that $g$ is a retraction in $\mathcal{C}$, a contradiction. In order to prove the statement, we need to prove that $\pi(\eta)$ remains an almost split exact triangle in $\overline{\mathcal{C}}$. We know that $\pi(h) \ne 0$ and that $\overline{\mathbb{S}}(\pi(M))= \pi(\overline{\mathbb{S}}(M))$ and $\pi(M)$ are indecomposable. Let $Z$ be indecomposable in $\overline{\mathcal{C}}$ and in $\mathcal{C}$ and consider $g: Z \to \pi(M)$ be a non-isomorphism. Consider a left fraction $(g_1, g_2)$ corresponding to $g$ where $g_2: M \to M'$ is in $\Sigma$. As usual, we may assume that $M'$ is indecomposable and that $M,M'$ correspond to similar arcs. But since $M$ corresponds to an ordinary arc, $g_2$ is an isomorphism in $\mathcal{C}$ and therefore can be assumed to be the identity morphism. Therefore, $g_1$ is not an isomorphism and hence not a retraction. Thus, $hg_1=0$ so $\pi(h)\pi(g_1)=\pi(h)g=0$, which proves that $\pi(\eta)$ is an almost split exact triangle. The statement follows from this (see Proposition I.2.3 in \cite{RVdB}).
\end{proof}

\begin{lemma} \label{Lemma_finitely_dim} Let $f: M \to N$ be a morphism in ${\rm mod}^{\rm fp}\mathcal{T}$ and let $V$ be a finitely presented subrepresentation of $M$ of dimension vector $g$. Then there are finitely many possible dimension vectors $[f(V')]$ for $V' \in {\rm Gr}_gM$.
\end{lemma}

\begin{proof}
Consider the quiver $\mathcal{G}(\mathcal{T}, S)$, which contains finitely many connected components, since it is finitely co-generated. Since $M,N,V$ are all $S$-uniform, we can partition the morphisms of a co-finite subset $S'$ of $S$ into finitely many subsets $S_1,\ldots, S_r$ in such a way that for each $i$, the morphisms in $S_i$ define a connected subquiver of $\mathcal{G}(\mathcal{T}, S)$ that are all $X$-isomorphisms for $X \in \{M,N,V\}$. Now, as we have already argued, this also implies that the latter property also holds for any $X=V'\in {\rm Gr}_gM$. Elementary considerations show that ${\rm dim}_k f(V')$ is constant on a given $S_i$. This implies the result.
\end{proof}

The following, known as the \emph{exchange formula}, is the last property needed for $X^{\mathcal{T}}$ to be a cluster character. The proof follows the same ideas as the proof of the exchange formula in \cite{P}. The difficulties, that we address below, can be summarized as follow. We need to make sure that the varieties defined in \cite{P} can also be interpreted as varieties in our setting, even for infinite dimensional representations with infinite dimensional subrepresentations. Also, we need to make sure that whenever the $2$-Calabi-Yau property is needed, then we are in the setting of Lemma \ref{locallyCY}. 

\begin{theorem}\label{ExchangePropertyCY} Let $M,N \in \overline{\mathcal{C}}$ where $M,N$ involve only ordinary arcs as direct summands, and such that $\Hom_{\overline{\mathcal{C}}}(M,N[1])$ is one dimensional. Let
$$M \to B_1 \to N \to M[1]$$
and
$$N \to B_2 \to M \to N[1]$$
be the corresponding non-split exact triangles. Then $$X^{\mathcal{T}}(M \oplus N) = X^{\mathcal{T}}(B_1) + X^{\mathcal{T}}(B_2).$$
\end{theorem}

\begin{proof}
Using the notations as in Lemma 4.2 of \cite{P}, consider an exact triangle
$$M \stackrel{i}{\to} B \stackrel{p}{\to} L \stackrel{\epsilon}{\to} M[1]$$
with morphisms $i_U: U \to M$ and $i_V: V \to L$ such that $\varphi(i_U)$ and $\varphi(i_V)$ are monomorphisms. We know that $\varphi(U), \varphi(V)$ are finitely presented. In the proof of Lemma 4.2 of \cite{P}, a submodule $E$ of $\varphi(B)$ is constructed as the image of a morphism in $\mmod^{\rm fp}\mathcal{T}$. Since the latter is abelian, we know that $E$ is finitely presented. Therefore, all modules involved in that proof are finitely presented. Therefore, Lemma 4.2 of \cite{P} is valid in our setting (by requiring that $E$ be finitely presented in their condition (i)).

The argument in the proof of Lemma 4.3 of \cite{P} extends in our setting, provided we make sure that his argument using the $2$-Calabi-Yau property can be adapted. It follows from Lemma \ref{locallyCY} that for $Z$ an object in $\overline{\mathcal{C}}$ corresponding to an ordinary arc, we have a functorial isomorphism
$$\Hom_{\overline{\mathcal{C}}}(Z,-) \cong D\Hom_{\overline{\mathcal{C}}}(-,\overline{\mathbb{S}}Z) \cong D\Hom_{\overline{\mathcal{C}}}(-,Z[2]),$$
and hence showing that the $2$-Calabi-Yau property
$$\Hom_{\overline{\mathcal{C}}}(Z,Z'[1]) \cong D\Hom_{\overline{\mathcal{C}}}(Z',Z[1])$$
holds whenever one of the objects $Z,Z'$ has only ordinary arcs as direct summands. Thanks to this, the full argument of Lemma 4.3 of \cite{P} works in our setting.

Let us fix two finitely presented dimension vectors $e$ and $f$ where $e$ is finitely presented in $\varphi(M)$ and $f$ is finitely presented in $\varphi(N)$. Let
$$X_{e,f} = \{E \subseteq \varphi(B_1) \mid [\varphi(i)^{-1}E]=e, \; [\varphi(p)E]=f\}$$
where $i,p$ denote the morphisms in our non-split exact triangle
$$M \stackrel{i}{\to} B_1  \stackrel{p}{\to} N \to M[1].$$
Since $e$ is finitely presented in $\varphi(M)$, it follows that  $\varphi(i)^{-1}(E)$ is finitely presented.  From Lemma \ref{Lemma_finitely_dim}, there are finitely many possible dimension vectors for $\varphi(i)\varphi(i)^{-1}(E)$ and such submodules of $\varphi(B_1)$ are finitely presented since they are images of finitely presented subrepresentations of $\varphi(M)$. Now, $E$ is an extension of $\varphi(i)\varphi(i)^{-1}(E)$ by a finitely presented subrepresentation of $\varphi(N)$ of dimension vector $f$. Therefore, $E$ is finitely presented and there are finitely many possible dimension vectors for $E$. For a finitely presented dimension vector $g$ in $\varphi(B_1)$, we let
$$X_{e,f}^g:={\rm Gr}_g\varphi(B_1)\cap X_{e,f}$$
We know that there are finitely many finitely presented $g$ such that $X_{e,f}^g$ is non-empty and therefore, $X_{e,f}$ is a projective variety identified with a finite union of subvarieties of Grassmannian varieties. We observe that
$${\rm Gr}_g\varphi(B_1) = \bigsqcup_{e,f}X_{e,f}^g$$
where the disjoint union runs over all pairs of finitely presented dimension vectors in $\varphi(M)$ and $\varphi(N)$, respectively. There are finitely many members of that union that are non-empty. Similarly, we define
$$Y_{e,f} = \{E \subseteq \varphi(B_2) \mid [\varphi(i')^{-1}E]=e, \; [\varphi(p')E]=f\}$$
where $i',p'$ denote the morphisms in our non-split exact triangle
$$N \stackrel{i'}{\to} B_2  \stackrel{p'}{\to} M \to N[1]$$
and for a finitely presented dimension vector $g$ in $\varphi(B_2)$, we let
$$Y_{e,f}^g:={\rm Gr}_g\varphi(B_2)\cap Y_{e,f}.$$
Now, we also have that $Y_{e,f}$ is a projective variety identified with a finite union of subvarieties of Grassmannian varieties. We also observe that
$${\rm Gr}_g\varphi(B_2) = \bigsqcup_{e,f}Y_{e,f}^g$$
where the disjoint union runs over all pairs of finitely presented dimension vectors in $\varphi(N)$ and $\varphi(M)$, respectively. For $e$ finitely presented in $\varphi(M)$ and $f$ finitely presented in $\varphi(N)$, we consider the map
$$X_{e,f} \sqcup Y_{e,f} \to {\rm Gr}_e\varphi(M) \times {\rm Gr}_f\varphi(N)$$
sending $E \in X_{e,f}$ to $(\varphi(i)^{-1}E, \varphi(p)E)$ and sending $E' \in Y_{e,f}$ to $( \varphi(p')E', \varphi(i')^{-1}E')$. By the analogue of Proposition 4.3 in \cite{P}, we know that this map is surjective and that any fiber lies in at most one of $X_{e,f}, Y_{e,f}$. Moreover, these fibers are affine spaces. Therefore, we get
\begin{eqnarray*}
\chi({\rm Gr}_e\varphi(M) \times {\rm Gr}_f\varphi(N)) & = & \chi(X_{e,f} \sqcup Y_{e,f})\\
& = & \chi(X_{e,f}) + \chi(Y_{e,f})\\
& = & \sum_g(\chi(X_{e,f}^g) + \chi(Y_{e,f}^g))
\end{eqnarray*}
where the last sum runs over all dimension vectors $g$. Note that for $X_{e,f}^g$ to be non-empty, one needs $g$ to be finitely presented in $\varphi(B_1)$. Similarly, for $Y_{e,f}^g$ to be non-empty, one needs $g$ to be finitely presented in $\varphi(B_2)$. So the last sum is really a finite sum.  Note that using similar arguments as in the proof of Lemma 5.1 in \cite{P}, we get
$${\rm coind}(g)-{\rm ind}(g) - {\rm coind}B_1 = {\rm coind}(e)-{\rm ind}(e) + {\rm coind}(f)-{\rm ind}(f) - {\rm coind}M - {\rm coind}N$$
and similarly, 
$${\rm coind}(g)-{\rm ind}(g) - {\rm coind}B_2 = {\rm coind}(e)-{\rm ind}(e) + {\rm coind}(f)-{\rm ind}(f) - {\rm coind}M - {\rm coind}N$$
Now, the product
$X^{\mathcal{T}}(M)X^{\mathcal{T}}(N)$ gives $$X^{-{\rm coind}(\varphi(M))-{\rm coind}(\varphi(N))}\sum_{e,f} \chi({\rm Gr}_e(\varphi(M)))\chi({\rm Gr}_f(\varphi(N)))X^{{\rm coind}(e) - {\rm ind}(e)+{\rm coind}(f) - {\rm ind}(f)}$$
which in turns gives
$$X^{-{\rm coind}(\varphi(M))-{\rm coind}(\varphi(N))}\sum_{e,f,g} (\chi(X_{e,f}^g) + \chi(Y_{e,f}^g))X^{{\rm coind}(e) - {\rm ind}(e)+{\rm coind}(f) - {\rm ind}(f)}.$$
Now, thanks to the above equalities on indices and coindices, the last sum splits as the sum of
$$ X^{-{\rm coind}(B_1)}\sum_{e,f,g} \chi(X_{e,f}^g)X^{{\rm coind}(g) - {\rm ind}(g)}$$ 
with
$$ X^{-{\rm coind}(B_2)}\sum_{e,f,g} \chi(Y_{e,f}^g)X^{{\rm coind}(g) - {\rm ind}(g)},$$ 
which gives the wanted result, since $\sum_{e,f}\chi(X_{e,f}^g) = \chi({\rm Gr}_g\varphi(B_1))$ and $\sum_{e,f}\chi(Y_{e,f}^g) = \chi({\rm Gr}_g\varphi(B_2))$.
\end{proof}

Now, we look at some examples. In particular, Example~\ref{fail-limit-arc} below shows that the condition that $M,N$ should correspond to ordinary arcs in Theorem~\ref{ExchangePropertyCY} cannot be omitted.

\begin{example}~\label{example-fountain}
In this example, we will look at the cluster character formula if we consider a cluster-tilting object $\mathcal{T}$ as shown in Figure~\ref{fig:exf}. 

\begin{figure}[H]
    \centering
    \includegraphics[scale=0.6]{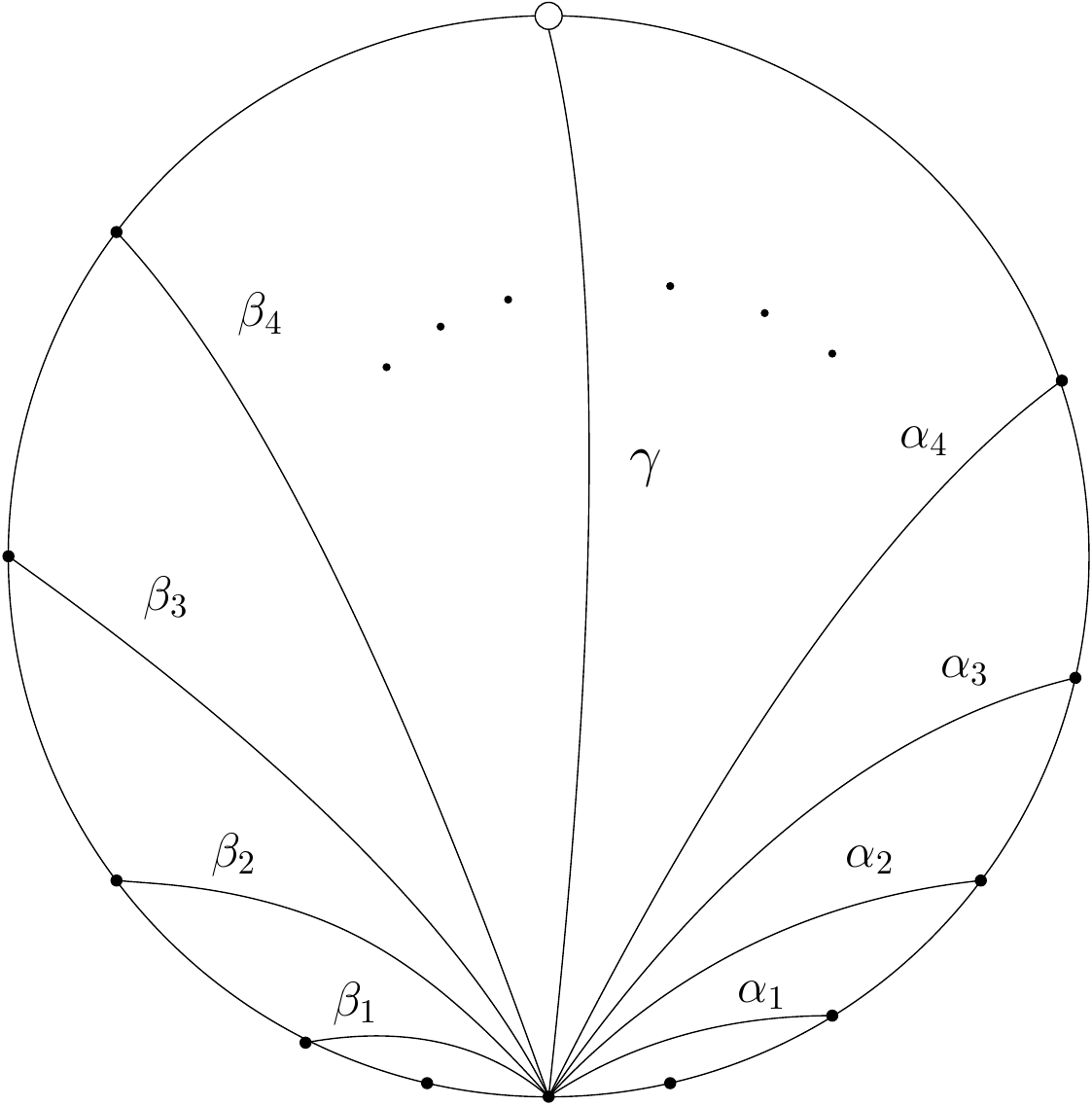}
    \caption{An illustration of a cluster-tilting object $\mathcal{T}$.}
    \label{fig:exf}
\end{figure}

Because of the equivalence in Theorem~\ref{EquivCategories}, we will think of an object in $\mathcal{\overline{\mathcal{C}}}/{\mathcal{T}[1]}$ as a representation over the category with the following quiver  $Q$
\[\xymatrix{1 & 2 \ar[l] & 3 \ar[l] &  \ar[l] \quad \cdots \quad \mathbf{z} \quad \cdots \quad & 4' \ar[l] & 3' \ar[l]  & 2' \ar[l]  & 1' \ar[l] }\]
where for each positive integer $i$, we have morphisms $i' \to z$ and $z \to i$ with the relations that all parallel morphisms are equal. In general, for an arc $\mu$, we denote by $M_\mu$ the indecomposable object corresponding to it. By an abuse of notation, we also denote the finitely presented module $\varphi(M_\mu)$ again by $M_\mu$. Consider the following projective representation at $z$.
\[M_{\gamma}=\xymatrix{k & k \ar[l] \cdots & k \ar[l] &  \ar[l] \cdots \quad \mathbf{k} \quad \cdots & 0 \ar[l] & \cdots 0 \ar[l]  & 0 \ar[l]}\]
where all maps between non-zero vector spaces are identity. Note that even though there is no arrow starting at $z$ in the quiver, the map $M_\gamma(z) \to M_\gamma(\alpha_i)$ is the identity for all $i$. 
In the cluster character formula, the sum runs over the finitely presented submodules. We consider all such submodules of $M_{\gamma}$ with the corresponding terms in the cluster character. We let $P_i = M_{\alpha_i}$, $P_{i'} = M_{\beta_i}$ and $P_z = M_\gamma$ for the corresponding indecomposable projective modules.

\begin{itemize}
\item We have the zero submodule which gives rise to the term $1$. 
\item Consider $M_{\alpha_1}$ with dimension vector $g=(1,0,\ldots,\mathbf{0},\ldots,0)$.
We have a triangle $0\rightarrow M_{\alpha_1}\rightarrow M_{\alpha_1} \rightarrow 0[1]$. Thus, ${\rm ind}(M_{\alpha_1})=[P_1]-[0]=[P_1]$.
Now, let us compute the ${\rm coind}(M_{\alpha_1})$. Identifying the objects or modules with their dimension vectors, we have a triangle \[(1,0,\ldots,\mathbf{0},\ldots,0)\rightarrow (1,1,\ldots,\mathbf{1},\ldots,1)\rightarrow (0,1,\ldots,\mathbf{1},\ldots,1)\rightarrow (1,0,\ldots,\mathbf{0},\ldots,0)[1]\] which is \[M_{\alpha_1} \rightarrow M_{\alpha_1}[2] \rightarrow M_{\alpha_2}[2] \rightarrow M_{\alpha_1}[1]\]
See Figure~\ref{fig:cc1} for the corresponding arc of $M_{\alpha_2}[2].$

\begin{figure}[H]
    \centering
    \includegraphics[scale=0.47]{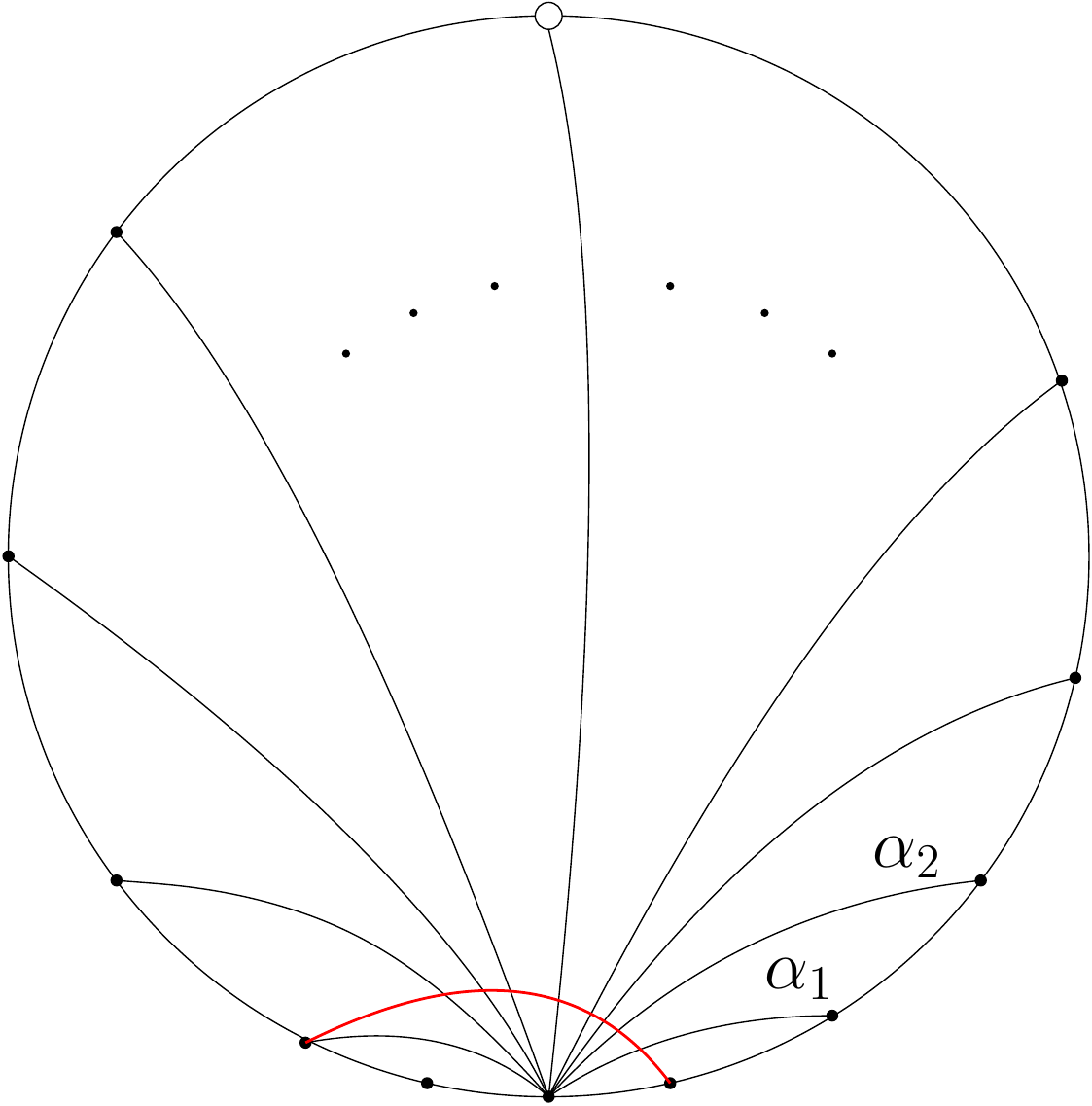}
    \caption{An illustration of $M_{\alpha_2}[2].$}
    \label{fig:cc1}
\end{figure}
We have ${\rm coind}(M_{\alpha_1}) = [P_1]-[P_2]$ and then $({\rm coind-ind})(M_{\alpha_1})=[P_1]-[P_2]-[P_1]$. This gives rise to $$x^{({\rm coind-ind})(M_{\alpha_1})}=x_1x_1^{-1}x_2^{-1}=x_2^{-1}.$$

\item Consider $M_{\alpha_i}$ with dimension vector $g=(1,\ldots,1,0,\ldots,\mathbf{0},\ldots,0)$. 
We have the triangle $0\rightarrow M_{\alpha_i}\rightarrow M_{\alpha_i} \rightarrow 0[1]$. So, ${\rm ind}(M_{\alpha_i})=[P_i]-[0]=[P_i]$.
Now, let us compute the ${\rm coind}(M_{\alpha_i})$. We have a triangle 
$$M_{\alpha_i} \rightarrow M_{\alpha_1}[2] \rightarrow M_{\alpha_{i+1}}[2] \rightarrow M_{\alpha_i}[1]$$
So, ${\rm coind}(M_{\alpha_i}) = [P_1]-[P_{i+1}]$ and $({\rm coind-ind})(M_{\alpha_i})=[P_1]-[P_{i+1}]-[P_i]$. Thus, we have $$x^{({\rm coind-ind})(M_{\alpha_i})}=x_1x_i^{-1}x_{i+1}^{-1}.$$

\item Consider $M_\gamma$ with dimension vector $g=(1,1,\ldots,\mathbf{1},\ldots,0)$. 
We have the triangle $0\rightarrow M_{\gamma}\rightarrow M_{\gamma} \rightarrow 0[1]$ which tells us ${\rm ind}(M_{\gamma})=[P_z]-[0]=[P_z].$ We also have a triangle 
$$M_{\gamma} \rightarrow M_{\alpha_1}[2] \rightarrow M_{\gamma}[2] \rightarrow M_{\gamma}[1]$$ that gives rise to 
${\rm coind}(M_{\gamma}) = [P_1]-[P_z].$
Thus, we have $({\rm coind-ind})(M_{\gamma})=[P_1]-2[P_z]$, and $$x^{({\rm coind-ind})(M_{\gamma})}=x_1x_{z}^{-2}.$$
\end{itemize}

Note that a module with dimension vector $(1,1,\ldots,\mathbf{0},\ldots,0)$ is not finitely generated, thus it is not finitely presented. Indeed, if $Z$ denotes such a module, then for a vertex $i$ of the quiver, we see that $\Hom(P_i,Z)$ is non-zero if and only if $i \in \mathbb{N}$. Therefore, if $P$ is finitely generated projective with a morphism $P \to Z$, then for $p$ large enough, the morphism $P(p) \to Z(p)$ is trivial.

In this example, the Euler characteristic of any Grassmannian of an indecomposable module $M$ with dimension vector g is always $0$ or $1$ because there is at most one choice for the finitely presented subrepresentation of $M$ of dimension vector $g$, thus the Grassmannian is either empty or a point. Therefore,

\begin{align*}
X^{\mathcal{T}}(M_{\gamma})=&  x_1^{-1}x_z(1+\sum_{n=1}^{\infty}x_1x_n^{-1}x_{n+1}^{-1}+x_1x_{z}^{-2}) \\
=& x_1^{-1}x_z+\sum_{n=1}^{\infty}x_zx_n^{-1}x_{n+1}^{-1}+x_{z}^{-1}.
\end{align*} 

\end{example}

\begin{example} Assume we have a cluster-tilting object $\mathcal{T}$ as in Example~\ref{example-fountain}. In this example, we will look at the exchange formula for the following triangles. 

\[M_{\alpha_3} \rightarrow M_{\beta_1} \rightarrow M_{\eta} \rightarrow M_{\alpha_3}[1]\]
\[M_{\eta} \rightarrow M_{\alpha_2} \oplus M_{\zeta} \rightarrow M_{\alpha_3} \rightarrow M_{\eta}[1]\]
where the modules correspond to the arcs shown in the Figure~\ref{triangles}.
\begin{figure}[H]
    \centering
    \includegraphics[scale=0.47]{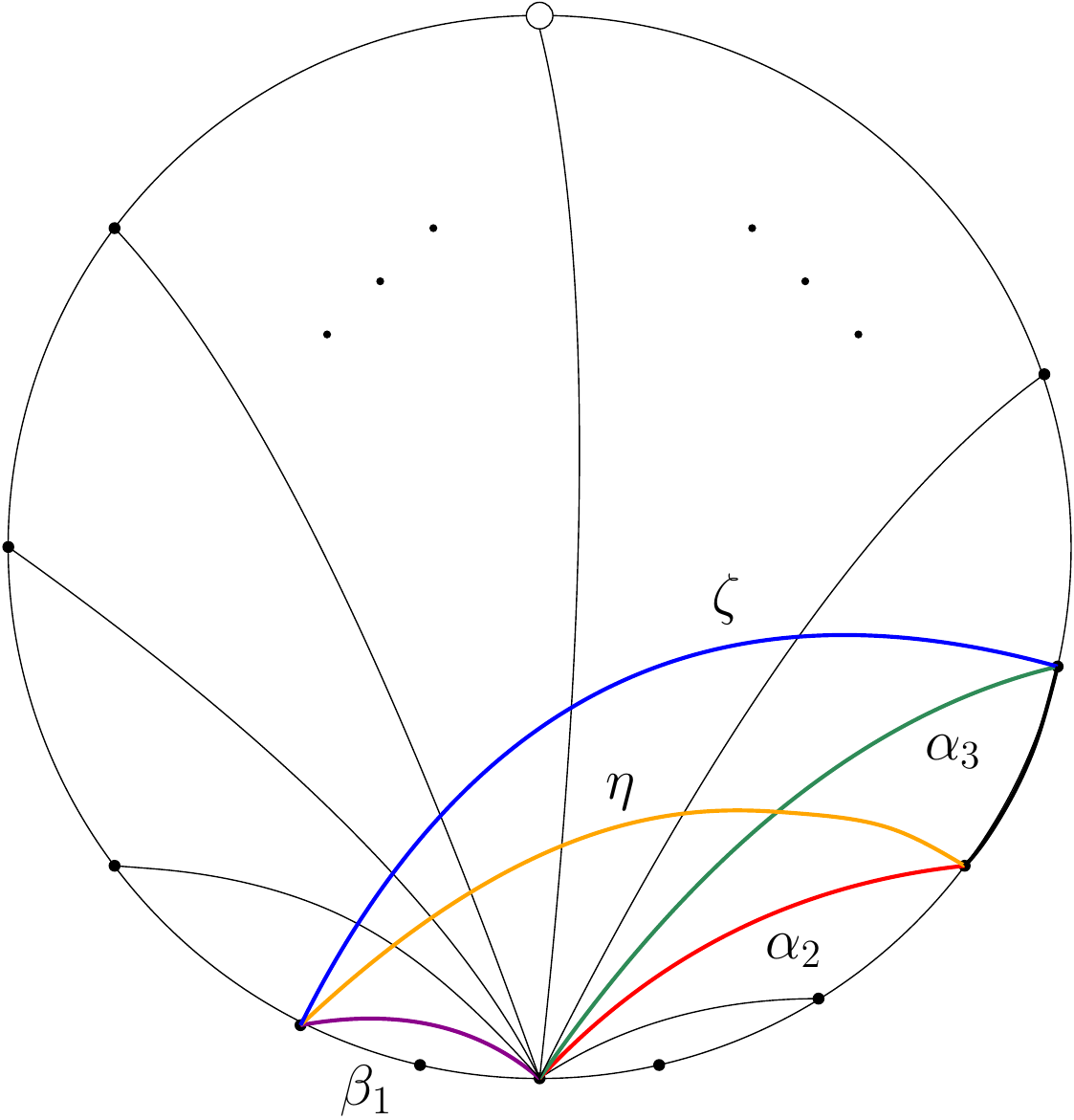}
    \caption{Arcs corresponding to objects in the triangles.}
    \label{triangles}
\end{figure}

We will illustrate and check that \[X^{\mathcal{T}}(M_{\eta})X^{\mathcal{T}}(M_{\alpha_3})=X^{\mathcal{T}}(M_{\alpha_2})X^{\mathcal{T}}(M_{\zeta})+X^{\mathcal{T}}(M_{\beta_1})\]

As in Example~\ref{example-fountain}, we compute 

\begin{align*}
X^{\mathcal{T}}(M_{\alpha_2})&=x_3\left(\frac{1}{x_1}+\frac{1}{x_1x_2}+\frac{1}{x_2x_3}\right)\\
X^{\mathcal{T}}(M_{\alpha_3})&=x_4\left(\frac{1}{x_1}+\frac{1}{x_1x_2}+\frac{1}{x_2x_3}+\frac{1}{x_3x_4}\right)
\end{align*}

For the computation of $X^{\mathcal{T}}(M_{\eta})$, we have three different type of finitely presented submodules (expressed with dimension vectors over the category with the quiver $Q$); 
\begin{enumerate}[(i)]
\item  $g=(0,0,0,1,1,\ldots,\mathbf{1},\ldots,1,1,\underbrace{0,\ldots,0}_{k-times})$ where $k$ is the number of zeros at the end of tuple and $k\geq 0$,\\
\item $g=(0,0,0,1,\ldots,\mathbf{1},\ldots,0)$,\\
\item $g=(0,0,0,\underbrace{1,\ldots,1}_{l-times},0,\ldots\mathbf{0},\ldots,0)$ where $l$ is the number of ones located after three zeros in the tuple and $l\geq 0$.
\end{enumerate}
For $(i)$, the corresponding terms are \[ \begin{cases}
    x_3x_4x_{1'}^{-1}      & \quad \text{if } k= 0\\
    x_3x_4x_{k'}^{-1}x_{(k+1)'}^{-1}  & \quad \text{otherwise}
  \end{cases}
\]
For $(ii)$, we get the term $x_3x_4x_{z}^{-2}$. 
Finally, in case $(iii)$, we get
\[ \begin{cases}
  1 & \quad \text{if } l= 0\\
x_3x_4x_{l+3}^{-1}x_{(l+4)}^{-1} & \quad \text{otherwise.}
  \end{cases}
\]

Since $x^{-{\rm coind}(M_{\eta})}= x_4^{-1}$,  we get 

\[X^{\mathcal{T}}(M_{\eta})=x_4^{-1}\left(1+\displaystyle \sum_{i=1}^{\infty} x_3x_4x_{i+3}^{-1}x_{(i+4)}^{-1} + x_3x_4x_{z}^{-2} + x_3x_4x_{1'}^{-1} +\displaystyle \sum_{i=1}^{\infty} x_3x_4x_{i'}^{-1}x_{(i+1)'}^{-1}\right)\]

The computation for $M_{\zeta}$ is almost identical to $M_{\eta}$. We have

\[X^{\mathcal{T}}(M_{\zeta})=x_5^{-1}\left(1+\displaystyle \sum_{i=1}^{\infty} x_4x_5x_{i+4}^{-1}x_{(i+5)}^{-1} + x_4x_5x_{z}^{-2} + x_4x_5x_{1'}^{-1} +\displaystyle \sum_{i=1}^{\infty} x_4x_5x	_{i'}^{-1}x_{(i+1)'}^{-1}\right)\]

By a similar computation, one gets 

\[X^{\mathcal{T}}(M_{\beta_1})=x_1^{-1}\left(1+\displaystyle \sum_{i=1}^{\infty} x_1x_{i}^{-1}x_{(i+1)}^{-1} + x_1x_{z}^{-2} + x_1x_{1'}^{-1} +\displaystyle \sum_{i=1}^{\infty} x_1x_{i'}^{-1}x_{(i+1)'}^{-1}\right)\]

After a careful calculation, it is straightforward to see that the exchange formula holds.
\end{example}

\begin{example}~\label{fail-limit-arc} In this example, we will show that the multiplication formula fails to hold for limit arcs. Consider the following configuration and the limit arc $\gamma$ drawn in red in Figure~\ref{gamma}.

\begin{figure}[H]
    \centering
    \includegraphics[scale=0.47]{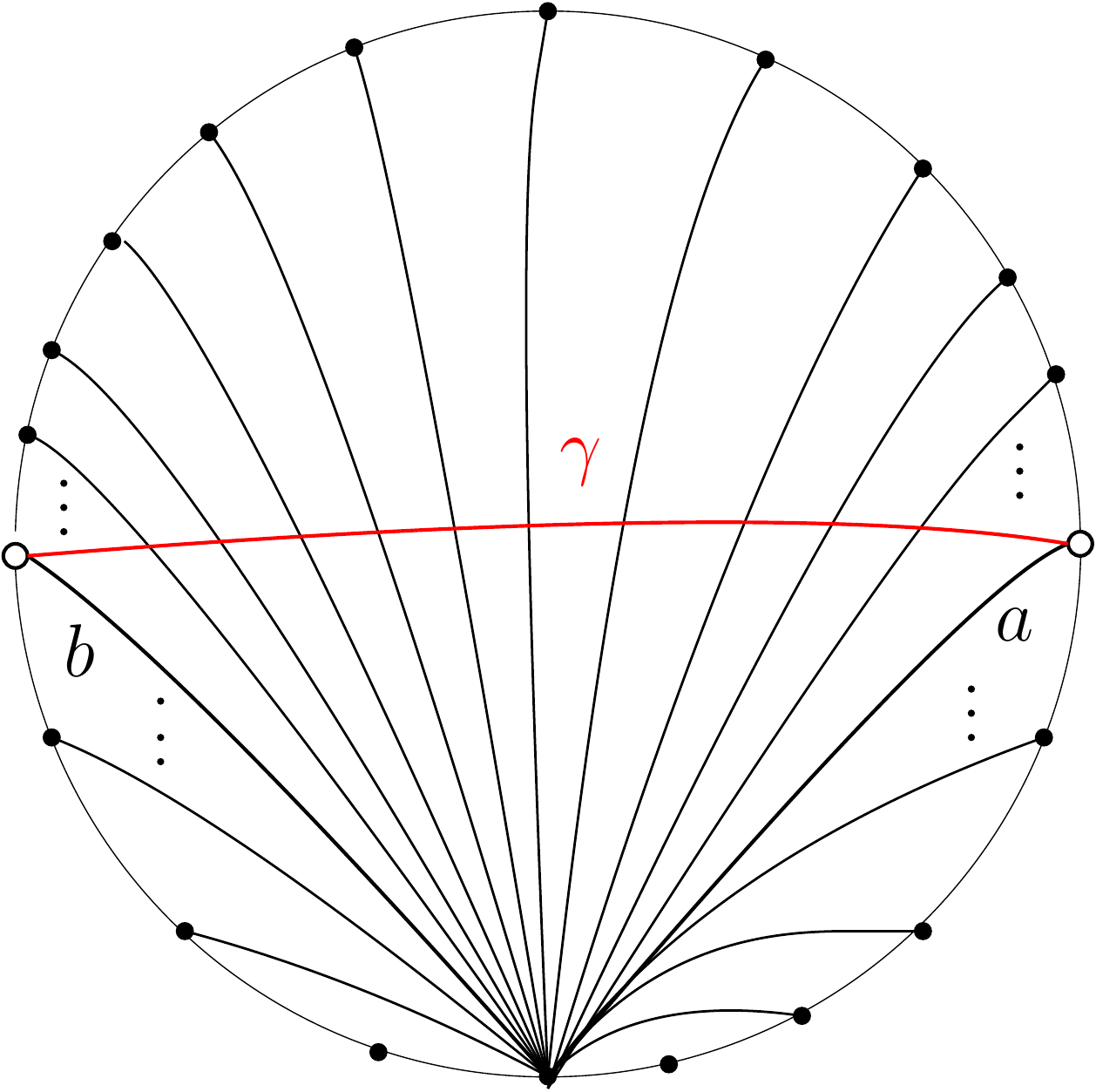}
    \caption{The limit arc $\gamma$ and an illustration of the underlying cluster-tilting object.}
    \label{gamma}
\end{figure}

As we did in Example~\ref{example-fountain}, we will think of objects in $\mathcal{\overline{\mathcal{C}}}/{\mathcal{T}[1]}$ as representations over the category whose quiver is 

$$\xymatrix{1 & 2 \ar[l] &  \ar[l]\quad \cdots {\bf a} \cdots \quad  &  \ar[l]\quad  \cdots {\bf b} \cdots \quad   & 2' \ar[l]  & 1' \ar[l]}$$

Recall that we identify the vertices of this quiver with the indecomposable objects of $\mathcal{T}$, up to isomorphism. There is a one-dimensional Hom-space from a given object to any other object on its left-hand side (by looking at the quiver). Let us consider the finitely presented module, denoted $M_\gamma$, with dimension vector 

$$\xymatrix{0 &  \ar[l]\quad \cdots {\bf 0} \cdots \quad  &   1 \ar[l] &  \ar[l] \quad  \cdots {\bf 1} \cdots \quad  & 0 \ar[l]  & 0 \ar[l]}$$ 
Let us consider the finitely presented subrepresentations of $M_{\gamma}$. Start with $M_{\gamma}$ itself. We have a triangle $P_a \rightarrow P_b \rightarrow M_{\gamma} \rightarrow P_a[1]$ where $P_a$, $P_b$ are the corresponding objects to arcs $a$ and $b$, respectively. So, the index of $M_{\gamma}$ is $[P_b]-[P_a]$. From the following triangle 
\[M_{\gamma} \rightarrow P_a[2] \rightarrow P_b[2]\rightarrow M_{\gamma}[1],\]
we compute the coindex as $[P_a]-[P_b]$. Thus, we get the term
$$x^{({\rm coind-ind})(M_{\gamma})}=x^2_ax_{b}^{-2}$$ where $x_a, x_b$ are the indeterminates corresponding to $P_a$ and $P_b$, respectively.

Consider the subrepresentation $M_k$ with dimension vector \[ (0,\cdots,{\bf 0},\cdots,1,\cdots,\underbrace{1}_{k\text{-th position}},0,\cdots,{\bf 0},\cdots,0)\] Similarly, one can compute the index as ${\rm ind}(M_k)=[P_k]-[P_a]$ and the coindex as ${\rm coind}(M_k)=[P_a]-[P_{k+1}].$ Therefore, we get the term
$$x^{({\rm coind-ind})(M_k)}=x^2_ax_{k}^{-1}x^{-1}_{k+1}.$$
Finally, the cluster character for $M_{\gamma}$ is
\[\left(1+x_a^2x_b^{-2}+\sum_{k=-\infty}^{\infty} x_a^2x_{k+1}^{-1}x_k^{-1}\right)x_a^{-1}x_b \]
Note that if the exchange formula were true for the non-split triangles
$$M_\gamma \to 0 \to M_\gamma \to M_\gamma[1] \; \text{(twice)},$$
then we would get $X^{\mathcal{T}}(M_\gamma)^2 = 2$. This  clearly shows that the exchange formula does not hold for the limit arcs.
\end{example}

\end{document}